\documentclass[11pt]{amsart}
\usepackage{amsmath,amssymb,amsthm,graphicx, url, verbatim}
\usepackage{color, enumerate, mathtools, float}
\usepackage{relsize}
\usepackage{thmtools}
\usepackage{thm-restate}
\usepackage{hyperref}
\usepackage{cleveref}
\usepackage{fnpct}
\usepackage[margin=1in]{geometry}

\theoremstyle{definition}
\newtheorem{defn}{Definition}[section]
\newtheorem{rem}[defn]{Remark}
\newtheorem{eg}[defn]{Example}
\theoremstyle{plain}
\newtheorem{thm}{Theorem}

\newtheorem{prop}[thm]{Proposition}

\newtheorem{lem}[defn]{Lemma} 
\newtheorem{cor}[thm]{Corollary}
\newtheorem{que}[defn]{Question}
\newtheorem{restate}{Theorem}
\newtheorem{restate2}{Theorem}
\hypersetup{
    colorlinks=true,
    linkcolor=blue,
    filecolor=magenta,      
    urlcolor=blue,
}

\newcommand{\Sl}{\text{sl}}
\newcommand{\msl}{\overline{\text{sl}}}

\newcommand{\Z}{\mathbb{Z}}
\newcommand{\sgn}{\text{sgn}}

\begin{document}
\title{Stability and triviality of the transverse invariant from Khovanov homology}
\author[D. Hubbard]{Diana Hubbard}
\author[C. Lee]{Christine Ruey Shan Lee}

\address[]{Department of Mathematics, Brooklyn College, Brooklyn, NY 11210}
\email[]{diana.hubbard@brooklyn.cuny.edu}

\address[]{Department of Mathematics and Statistics, University of South Alabama, Mobile AL 36608}
\email[]{crslee@southalabama.edu}

\graphicspath{{Pics/}}

\thanks{Hubbard was partially supported by an NSF RTG grant 1045119 and an AMS-Simons travel grant during the completion of this work. Lee was partially supported by an NSF postdoctoral fellowship DMS 1502860 and an NSF grant DMS 1907010.}

\begin{abstract}

We explore properties of braids such as their fractional Dehn twist coefficients, right-veeringness, and quasipositivity, in relation to the transverse invariant from Khovanov homology defined by Plamenevskaya for their closures, which are naturally transverse links in the standard contact $3$-sphere. For any $3$-braid $\beta$, we show that the transverse invariant of its closure does not vanish whenever the fractional Dehn twist coefficient of $\beta$ is strictly greater than one. We show that Plamenevskaya's transverse invariant is stable under adding full twists on $n$ or fewer strands to any $n$-braid, and use this to detect families of braids that are not quasipositive. Motivated by the question of understanding the relationship between the smooth isotopy class of a knot and its transverse isotopy class, we also exhibit an infinite family of pretzel knots for which the transverse invariant vanishes for every transverse representative, and conclude that these knots are not quasipositive.
\end{abstract}
\maketitle

\section{Introduction}

Khovanov homology is an invariant for knots and links smoothly embedded in $\mathbb{S}^{3}$ considered up to smooth isotopy. It was defined by Khovanov in \cite{khovanov1999categorification} to be a categorification of the Jones polynomial. Khovanov homology and related theories have had numerous topological applications, including a purely combinatorial proof due to Rasmussen {\cite{rasmussen2010khovanov} of the Milnor conjecture (for other applications see for instance \cite{ng2005legendrian} and \cite{kronheimer2011khovanov})}. In this paper we will consider Khovanov homology calculated with coefficients in $\mathbb{Z}$ and reduced Khovanov homology with coefficients in $\Z/2\Z$. 

Transverse links in the contact $3$-sphere are links that are everywhere transverse to the standard contact structure induced by the contact form $\xi_{st} = dz + r^{2} d \theta$. Bennequin proved in \cite{bennequin1983entrelacements} that every transverse link is transversely isotopic to the closure of some braid. Furthermore, Orevkov and Shevchisin \cite{orevkov2003markov}, and independently Wrinkle \cite{wrinkle2002markov}, showed that there is a one-to-one correspondence between transverse links (up to transverse isotopy) and braids (up to braid relations, conjugation, and positive stabilization). Hence we can study transverse links by studying braids. Plamenevskaya used this to observe in \cite{Pla06} that Khovanov homology can be used to define an invariant of transverse links. Given a braid $\beta$ whose closure $\widehat{\beta}$ is transversely isotopic to a transverse link $K$, she showed that there is a distinguished element $\widetilde{\psi}(\beta)$ in the Khovanov chain complex $CKh(\widehat{\beta})$ whose homology class $\psi(\beta)$ in the Khovanov homology of $K$ is a transverse invariant that encodes the classical self-linking number. Plamenevskaya also defined a version of this transverse invariant in reduced Khovanov homology, which we will denote as $\psi'(\beta)$, see Section \ref{subsec:redpsi}.

A transverse invariant is called \emph{effective} if it can distinguish between a pair of smoothly isotopic but not tranversely isotopic links with the same self-linking number. It is an open question whether $\psi$ is an effective transverse invariant. Several efforts \cite{lipshitz2015transverse, wu2008braids, hubbard2016annular, collari2017transverse} have been made to both understand the effectiveness of $\psi$ and to define new invariants related to $\psi$ in the hope that one of these would be effective. Thus far these efforts have not yielded any transverse invariants arising from Khovanov-type constructions that are known to be effective or not.  However, $\psi$ has other applications. For instance, $\psi$ and one of its refinements provide new solutions to the word problem in the braid group \cite{baldwin2015categorified}, \cite{hubbard2016annular}.

One of the goals in this paper is to explore the following question:

\begin{que}\label{psiproperties} Given a transverse knot $K$, what properties of $K$ does $\psi(K)$ detect?
\end{que}

In \cite{plamenevskaya2015transverse}, Plamenevskaya explored Question \ref{psiproperties} for another transverse invariant, $\hat\theta$, arising from knot Floer homology \cite{ozsvath2008legendrian}, which she computed using $\mathbb{Z}/2\mathbb{Z}$ coefficients. In contrast to $\psi$, $\hat\theta$ is known to be effective \cite{ng2008transverse}. Plamenevskaya showed that given a transverse link $K$ with a braid representative $\beta$, the behavior of $\hat\theta(K)$ is related to dynamical properties of $\beta$ when $\beta$ is viewed as acting on the $n$-punctured disk $D_{n}$.

\begin{thm}\label{RVPlamenevskaya} [\cite{plamenevskaya2015transverse}, Theorem 1.2] Suppose $K$ is a transverse knot that has a $3$-braid representative $\beta$. Every braid representative of $K$ is right-veering if and only if $\hat\theta(K) \neq 0$. 
\end{thm}

\begin{thm}\label{FDTCPlamenevskaya} [\cite{plamenevskaya2015transverse}, Theorem 1.3] Suppose $K$ is a transverse knot that has an n-braid representative $\beta$ with fractional Dehn twist coefficient $\tau(\beta) > 1$. Then $\hat\theta(K) \neq 0$.
\end{thm}

Informally, the fractional Dehn twist coefficient of a braid $\beta$ measures the amount of rotation $\beta$ effects on the boundary of the punctured disk $D_{n}$, see Section \ref{subsec:defnsofRVQPFDTC}. The fractional Dehn twist coefficient can be defined in general for elements in the mapping class group of any surface $\Sigma$ with a single boundary component. As all right-veering braids have fractional Dehn twist coefficient greater than or equal to $0$, see \cite{etnyre2015monoids}, Theorem \ref{FDTCPlamenevskaya} allows us to conclude that, roughly, ``most" right-veering braids have non-vanishing $\hat\theta$. Theorem \ref{FDTCPlamenevskaya} is similar in flavor to a previous result about contact structures: work of Honda, Kazez, and Mati{\'c} in \cite{honda2008right}, together with that of Ozsv{\'a}th and Szab{\'o} in \cite{ozsvath2004holomorphic} proves that a contact structure supported by an open book decomposition with connected binding where the pseudo-Anosov monodromy has fractional Dehn twist coefficient greater than or equal to one has non-vanishing Heegaard Floer twisted contact invariant.

In this paper we first consider the behavior of $\psi$ and $\psi'$ with respect to the property of being right-veering and the fractional Dehn twist coefficient, and study the extent to which the analogous statements of Theorem \ref{RVPlamenevskaya} and \ref{FDTCPlamenevskaya} hold. We note that several facts were already known relating the behavior of $\psi$ to properties of braids: that $\psi$ does not vanish for transverse links that have a quasipositive braid representative \cite{Pla06} and that it does vanish for transverse links with non-right-veering braid representatives \cite{baldwin2015categorified} and links with $n$-braid representatives that are negative stabilizations of an $(n-1)$-braid \cite{Pla06}. See Section \ref{subsec:defnsofRVQPFDTC} for more background. These properties also hold for $\psi'$.

A calculation (see Section  \ref{sec:appsofgenstability}) shows that the statement corresponding to Theorem \ref{RVPlamenevskaya} is not true for $\psi$ (nor $\psi'$): there exist right-veering $3$-braids, namely the family $\Delta^{2}\sigma_{2}^{-k}$ for sufficiently large $k \in \mathbb{N}$, for which $\psi$ and $\psi'$ vanish on their closures.

However, we show that the analogue of Theorem
\ref{FDTCPlamenevskaya} holds for 3-braids:

\begin{thm}\label{3braidFDTC} Suppose $K$ is a transverse knot that has a $3$-braid representative $\beta$ with fractional Dehn twist coefficient $\tau(\beta) > 1$. Then $\psi(K) \neq 0$ (when computed over $\mathbb{Q}$, $\mathbb{Z}$, and $\mathbb{Z}/2\mathbb{Z}$ coefficients) and $\psi'(K) \neq 0$.
\end{thm}

Our second result shows a ``stability" property of $\psi$ under adding a sufficient number of negative or positive twists on any number of strands to an arbitrary braid word. 

\begin{thm} \label{thm:stab} Let $L$ be any closed braid $\widehat{\beta \alpha^{\pm}}$ with $\beta$ of strand number $b$ and $\alpha^{\pm}$ of strand number $2\leq a<b$ consisting of positive/negative sub-full twists 
$$\alpha^{\pm} =  (\sigma_{i}^{\pm} \sigma_{i+1}^{\pm}\cdots \sigma^{\pm}_{i+a-2})^a. $$ where $1 \leq i \leq b-a+1$.  Denote by $L^m_{\pm}$ the closed braid $\widehat{\beta (\alpha^{\pm})^m}$. There is some $N$ for which we have that for all $m > N$, $\psi(L_{\pm}^m) =0$  if and only if  $\psi(L_{\pm}^{m+1})=0$. Similarly,  $\psi'(L_{\pm}^m) = 0$ if and only if  $\psi'(L_{\pm}^{m+1})=0$.
\end{thm}

The behavior of $\psi$ has several topological applications \cite{Pla06, baldwin2010khovanov}. However, in many cases $\psi$ is very difficult to compute. Theorem \ref{thm:stab} implies that, given a vanishing/non-vanishing result for $\psi$ of a braid closure, we can immediately extend it to an infinite family obtained by adding sub-full twists.
Furthermore, we have concrete bounds for $N$ based on the number of negative or positive crossings in $\beta$. Throughout the rest of the paper, we will refer to these sorts of ``full" twists on fewer than the full number of strands as ``sub-full" twists. This theorem means that after adding a large enough number of sub-full twists, the transverse invariant stabilizes. This echoes the results by \cite{CK05} which demonstrates a stability behavior of the Jones polynomial of a braid under adding full twists on any number of strands, and \cite{Sto07} which considers stability in the Khovanov homology of infinite torus braids. 

Note that Theorem \ref{thm:stab} is immediate if the twists we add are full twists on $b$ strands instead of sub-full twists. Indeed, adding sufficiently many positive full twists to any braid will result in a positive braid, which has non-vanishing $\psi$. Similarly, adding sufficiently many negative full twists to any braid will result in a non-right-veering braid, which has vanishing $\psi$, and the non-right-veeringness will be preserved under adding even more negative twists. 

As an application of Theorem \ref{thm:stab}, we find many examples of braids which are right-veering but not quasipositive, see Section \ref{sec:appsofgenstability}, which shows that $\psi$ may be used to detect quasipositivity. In general, it is of interest - particularly to contact geometers - to detect braids that are right-veering but not quasipositive.  Indeed, this was a main theme of the work of Honda, Kazez, and Mati\'c in \cite{honda2008right}: the idea is that the difference between right-veering and quasipositive braids reflects the difference between tight and Stein fillable contact structures. These braids are often constructed by adding sub-full negative twists to right-veering braids.

For $4$-braids we have:

\begin{prop}\label{4braidexample} There exist families of $4$-braids, $\alpha_{k} = \Delta^{2}\sigma_{2}^{-k}$,  $\beta_{k} = \Delta^{2}\sigma_{3}^{-k}$, and $\eta_{k} = \Delta^{2}(\sigma_{2}\sigma_{3})^{-k}$ such that for $k \in \mathbb{N}$:
\begin{enumerate}
\item $\tau(\alpha_{k}) = \tau(\beta_{k}) = \tau(\eta_{k}) = 1$.
\item $\hat\theta(\widehat{\alpha_{k}}) \neq 0$, $\hat\theta(\widehat{\beta_{k}}) \neq 0$, $\hat\theta(\widehat{\eta_{k}}) \neq 0$.  (Plamenevskaya, proof of Theorem \ref{FDTCPlamenevskaya} in \cite{plamenevskaya2015transverse}).
\item For $k \geq 12$, $\alpha_{k}$, $\beta_{k}$, and $\eta_{k}$ are not quasipositive \footnote{This is due to a simple calculation of the writhe. These braids may be not quasipositive for some values of $k < 12$.}, and $\psi/\psi'(\alpha_{k}) \neq 0$, $\psi/\psi'(\beta_{k}) \neq 0$, but $\psi/\psi'(\eta_{k})=0$. 
 
\end{enumerate}
\end{prop}

Proposition \ref{4braidexample} allows us to conclude that an infinite collection of non-quasipositive $4$-braids with fractional Dehn twist coefficient greater than one have non-vanishing $\psi$. Using functoriality allows us to conclude that any braid that has a word of the form $\Delta^{2} \sigma$ where $\sigma$ contains only positive powers of $\sigma_{1}$ and $\sigma_{2}$ but arbitrarily many negative powers of $\sigma_{3}$ has non-vanishing $\psi$. Many such braids are not quasipositive, and thus $\psi$ does not primarily detect quasipositivity. We remark that in general, it is not known whether sufficiently large fractional Dehn twist coefficient guarantees non-vanishing $\psi$. For instance, it may be that $n$ being large enough for the braid $\triangle^{2n}(\sigma_2\sigma_3)^{-k}$ guarantees $\psi(\triangle^{2n}(\sigma_2\sigma_3)^{-k}) \not=0$ regardless of $k$.  The behavior of $\psi$ and $\psi'$ for the family $\eta_{k}$ allows us to conclude that $\hat\theta$ and $\psi$ can differ for braids with more than three strands.

A different perspective on Question \ref{psiproperties} is whether one can characterize \emph{smooth} link types for which \emph{every} transverse representative has vanishing $\psi$. In some sense, this question is asking about properties of smooth link types in which $\psi$ has no chance of distinguishing between distinct transverse representatives. Notice that every link type has infinitely many distinct transverse representatives, and \emph{some} transverse representative for which $\psi$ vanishes. For instance, one can always negatively stabilize a braid $\beta$ to yield $\beta'$ with $\psi(\beta')=0$, another braid representative of the link represented by $\hat{\beta}$.  The transverse link $\hat{\beta'}$ is not transversely isotopic to $\hat{\beta}$ as their self-linking numbers differ by two \cite{orevkov2003markov}, \cite{wrinkle2002markov}.

One way to explore this question is by examining the relationship between the Khovanov homology of a smooth link type $L$ and its maximal self-linking number, see Definition \ref{d.msl}.  This quantity is of natural interest since it provides bounds on several topological link invariants, including the slice genus, see  \cite{rudolph1993quasipositivity} and \cite{Ng08}. The distinguished element $\widetilde{\psi}$ in the Khovanov chain complex of a transverse representative $\hat\beta$ of $L$ lives in homological grading $0$ and quantum grading the self-linking number of $\hat\beta$. We have the following immediate observation.

\begin{rem} \label{rem:msl0} Suppose the maximal self-linking number of $L$ is $n$. If every nontrivial homology class in homological grading $0$ of the Khovanov homology of $L$ has quantum grading strictly greater than $n$, then $\psi$ vanishes for every transverse representative of $L$. 
\end{rem}

\begin{eg}\label{Ngexample}  According to Proposition 4 of \cite{ng2012arc}, the mirror of the knot $11n33$, which we denote $\overline{11n33}$, has maximal self-linking number $-7$. Using the Khovanov polynomial for $\overline{11n33}$ in KnotInfo \cite{knotinfo}, we see that in homological grading $0$, the Khovanov homology of $\overline{11n33}$ is empty for $q$-grading less than $-3$. We can conclude that every transverse representative of $\overline{11n33}$ has vanishing $\psi$.
\end{eg} 

We describe an infinite family of 3-tangle pretzel knots for which $\psi=0$ for every transverse representative by Remark \ref{rem:msl0}. In particular, we give conditions on the parameters of pretzel knots that guarantee that $\psi$ vanishes for every transverse representative of $L$ using the bound on the maximal self-linking number by Franks-William \cite{FW87}, Morton \cite{Mor86}, and Ng \cite{ng2005legendrian} from the HOMFLY-PT polynomial.

\begin{thm} \label{thm:homflypretzel}
Let $K=P(r, -q,  -q)$ be a pretzel knot with $q>0$ odd and $r\geq 2$ even, then $\psi = 0$ for every transverse link representative of $K$. 
\end{thm}

We conclude that every such pretzel knot has no quasipositive braid representatives and hence is not quasipositive, see Corollary \ref{cor:pretzelQP} and the following discussion. Preliminary computational evidence based on the braid representatives of $K$ from the program by Hilary Hunt available at \url{https://tqft.net/web/research/students/HilaryHunt/} \cite{Hilaryhunt} implementing the Yamada-Vogel algorithm suggests that the fractional Dehn twist coefficient of this family may always be less than or equal to one. This is also true for another braid representative $\sigma_1\sigma_1\sigma_1\sigma_2^{-1}\sigma_1^{-1}\sigma_1^{-1}\sigma_3\sigma_2\sigma_2\sigma^{-1}_4\sigma_3\sigma^{-1}_4$\cite{knotinfo} for $\overline{11n33}$ from Example \ref{Ngexample} which in fact has fractional Dehn twist coefficient equal to $0$. It seems possible that this is a characterization of links for which every single braid representative has $\psi=0$, and we will address this question in the future.

\subsection*{Organization} This paper is organized as follows: we give the preliminary background on Khovanov homology, reduced Khovanov homology, and the transverse invariant in Section \ref{sec:bg}, and we prove Theorem \ref{3braidFDTC} in Section \ref{sec:3-braids}. Theorem \ref{thm:stab} is proven in Section \ref{sec:genstability}, and we collect a few examples and prove Proposition \ref{4braidexample} in Section \ref{sec:appsofgenstability}. Finally, we prove Theorem \ref{thm:homflypretzel} in Section \ref{maxselflinking}, where the necessary results on the maximal self-linking number and the HOMFLY-PT polynomial are summarized.

\subsection{Acknowledgments} Both authors would like to thank John Baldwin for his computer program computing $\psi'$ available at \url{https://www2.bc.edu/john-baldwin/Programs.html} as it was very helpful. We would also like to thank Hilary Hunt for her computer program implementing the Yamada-Vogel algorithm. In addition both authors would like to acknowledge the database KnotInfo, and KnotAtlas, which makes available the KnotTheory package. The first author would also like to acknowledge helpful private correspondence with Olga Plamenevskaya and Peter Feller and several interesting discussions with Adam Saltz and John Baldwin. The second author would like to thank Matt Hogancamp for interesting conversations on the stable Khovanov homology of torus knots. 

\section{Background} \label{sec:bg}
In this section we will set our conventions and briefly review Khovanov homology, the transverse invariant defined by Plamenevskaya \cite{Pla06}, and standard tools used for computing the invariant. We will also review the definitions of quasipositivity, right-veering, and the fractional Dehn twist coefficient.

\subsection{Khovanov homology} The readers may refer to \cite{Bar02}, \cite{Tur17} for excellent introductions to the subject. 
Given a crossing in an oriented link diagram $D$, a Kauffman state chooses the $0$-resolution or the $1$-resolution as depicted in the following figure, which replaces the crossing by a set of two arcs.
\begin{figure}[ht]
\centering
\def\svgwidth{.3\columnwidth}
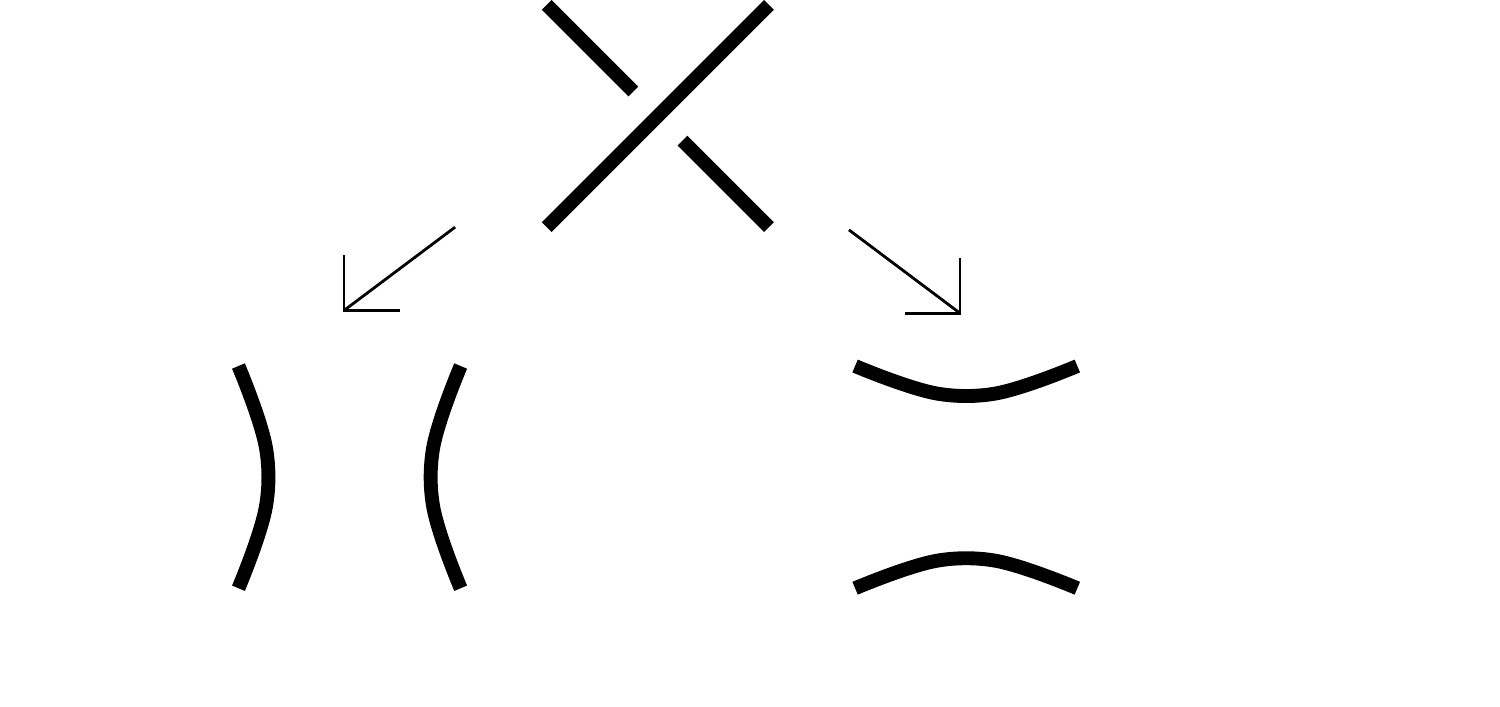
\end{figure} 
Number the crossings of $D$ from $1,\ldots, n$. Each Kauffman state $\sigma$ on $D$ can be represented by a string of 0 and 1's in $\{0, 1\}^n$ where $0$ at the $i$th position means that the $0$-resolution is chosen at the $i$th crossing of $D$ and similarly $1$ at the $i$th position means that the $1$-resolution is chosen at the $i$th crossing. 

The bi-graded chain complex $CKh(D)$ is generated by a direct sum of $\mathbb{Z}$-vector spaces associated to a Kauffman state $\sigma$. 
\[CKh(D) := \bigoplus_{\sigma} CKh(D_{\sigma}), \]
where $CKh(D_{\sigma})$ is defined as follows. Let $s_{\sigma}(D)$ be the set of disjoint circles resulting from applying the Kauffman state $\sigma$ to $D$, and let $|s_{\sigma}(D)|$ be the number of circles. Then 
\[CKh(D_{\sigma}) := V^{\otimes|s_{\sigma}(D)|}, \]
where $V$ is the free graded $\mathbb{Z}$-module generated by two elements $v_-$ and $v_+$ with grading $p$ such that $p(v_{\pm}) = \pm 1$. The grading is extended to the tensor product by the rule $p(v\otimes v') = p(v) + p(v')$. 

Let $r(\sigma)$ be the number of 1's in the string in $\{0, 1\}^n$ representing a Kauffman state $\sigma$. Two gradings $i,  j$ are defined on $CKh(D)$ as follows. The \emph{homological grading} $i$ is defined by $i(v) = r(\sigma)- n_-(D)$, where $\sigma$ is the state giving rise to the vector space $V^{\otimes |s_{\sigma}(D)| }$ containing $v$, while the \emph{quantum grading} $j$ is defined by $j(v) = p(v) + i(v) + n_+(D)-n_-(D)$, where $n_{+}(D)$ and $n_{-}(D)$ are the number of positive and negative crossings in $D$, respectively.

We shall indicate the $\Z$-vector space with bi-grading $(i,j)$ in $CKh(D)$ as $CKh^i_j(D)$. 
For the differential of the chain complex, we first define a map $d_c$ on $CKh(D)$ from $\sigma$ to $\sigma_c$:
\[d_c: V^{\otimes|s_{\sigma}(D)|} \rightarrow V^{\otimes|s_{\sigma_c}(D)|},  \]
where $\sigma$ and $\sigma_c$ differ in their resolution at exactly one crossing $c$ where $\sigma$ chooses the $0$-resolution and $\sigma_c$ chooses the $1$-resolution. From $\sigma$ to $\sigma_c$, either two circles merge into one or a circle splits into two. In the first case, the map $d_c$ contracts $V\otimes V$, representing the pair of circles in $s_{\sigma}(D)$, to $V$, representing the resulting circle in $s_{\sigma_c}(D)$  by the merging map $m$ as defined below. 
\begin{align*}
&m(v_+\otimes v_+) = v_+ \\ 
&m(v_+\otimes v_-)= m(v_-\otimes v_+) = v_- \\ 
&m(v_-\otimes v_-) = 0.
\end{align*}

In the other case where a circle splits into two, $d_c$ is given by the splitting map $\triangle$ taking $V\rightarrow V\otimes V$ as follows. 
\begin{align*}
&\triangle(v_+) = v_+\otimes v_- + v_-\otimes v_+ \\ 
&\triangle(v_-) = v_-\otimes v_-.
\end{align*}

Now on $CKh^i_{*}(D_{\sigma})$ the differential $d$ is defined by 
\[d = \sum_{c \text{ \ crossing in $D$ on which $\sigma$ chooses the 0-resolution}} (-1)^{\sgn(\sigma, \sigma_c)}d_c.\] 
We extend $d$ by linearity. The sign $\sgn(\sigma, \sigma_c)$ is chosen so that $d\circ d = 0$; for instance one can choose that $\sgn(\sigma, \sigma_c)$ is the number of 1's in the string representing $\sigma$ before $c$. The resulting homology groups $Kh(D)$ are independent of this choice. 

Khovanov \cite{khovanov1999categorification} defined and showed that $Kh(D)$ is independent of the diagram chosen for the link $L$, so $Kh(L) = Kh(D)$ is a link invariant. 
\subsubsection{Reduced Khovanov homology with $\Z/2\Z$-coefficients} \label{ss.reduced}
Given an oriented link diagram $D$ with a marked point on a link strand, consider the Khovanov complex $CKh(D)$. A circle in a state $s_{\sigma}(D)$ is marked if it contains the marked point. We denote the marked circle of a state $\sigma$ by $\sigma(m)$. Consider the sub-complex 
\[CKh(D, -):=  \bigoplus_{\sigma} V_1 \otimes \cdots \otimes V_{\sigma(m)-1} \otimes  v_- \otimes V_{\sigma(m)+1}\cdots \otimes V_{|s_\sigma(D)|}, \] where $v_-$ is the element in the vector space assigned to $\sigma(m)$, and let $\overline{CKh}(D)$ be the quotient complex $CKh(D)/CKh(D, -)$. Reduced Khovanov homology is generated by $\overline{CKh}(D)$ with the differential $d$ of $CKh(D)$ descending to the differential $d'$ on the quotient complex. When the vector space $CKh(D)$ is generated as a direct sum of $\Z/2\Z$-vector spaces, $\overline{Kh}(D)$ is an invariant of the link represented by $L$, independent of the placement of the marked point \cite{Kho03}. Thus, we will denote by $\overline{Kh}(L)$ the reduced Khovanov homology of $L$ with $\Z/2\Z$-coefficients. We will also adopt the convention of shifting the quantum grading by -1, see for example \cite{baldwin2010khovanov} so that the homology of the unknot is at grading $i=0, j=0$ in this theory.

\subsection{The transverse element} Here we follow the conventions of Plamenevskaya \cite{Pla06} except for a minor change in notation.  In her paper the bi-grading is indicated as $Kh_{i, j}$, whereas in this paper the homological grading is placed on top as $Kh_j^i$. 

Let $\beta$ be a braid representative of a link $L$ giving a closed braid diagram $\hat{\beta}$ of $L$. Consider the \emph{oriented resolution}, the Kauffman state $\sigma_{\beta}$ of $\hat{\beta}$ where we take the 0-resolution for each positive crossing and the 1-resolution for each negative crossing. 

\begin{defn} The \emph{transverse invariant} of a closed braid representative $\widehat{\beta}$ of $L$, denoted by $\psi(\hat{\beta})$, is the homology class in $Kh(L)$ of the following element in the vector space associated to $\sigma_{\beta}$:
\[ \widetilde{\psi}(\beta) := v_- \otimes v_- \otimes \cdots \otimes v_- \in V^{\otimes|s_{\sigma_{\beta}}(\hat{\beta})|}=CKh(\hat{\beta}_{\sigma_{\beta}}).\]
\end{defn} 
Plamenevskaya has shown that this is, up to sign, a well-defined homology class in $Kh(L)$ \cite[Proposition 1]{Pla06} under transverse link isotopy. Thus $\psi(\hat{\beta})$ is a transverse link invariant which lies in $ Kh^0_{\Sl(\hat{\beta})}(L)$, where $\Sl(\hat{\beta})$ is the \emph{self-linking number} of a transverse link represented by a closed braid $\hat{\beta}$ defined as follows. 
\begin{defn}\label{defn:selflinking} The \emph{self-linking number} of the transverse link $\hat{\beta}$ is given by
\[\Sl(\hat{\beta})= -b + n_+(\hat{\beta})-n_{-}(\hat{\beta}).\]
Note $b$ is the strand number of $\beta$.
\end{defn}
To simplify the notation, we will omit the hat $`` \ \hat{} \ "$ in $\psi(\hat{\beta})$ and simply write $\psi(\beta)$.

\subsubsection{The reduced version}\label{subsec:redpsi}
In the reduced setting, the transverse invariant of a closed braid representative $\hat{\beta}$ of $L$, denoted by $\psi'(\beta)$, is the homology class (up to sign) in $\overline{Kh}(L)$ of the following element in the vector space associated to $\sigma_{\beta}$ 
\[ \widetilde{\psi}'(\beta) := v_- \otimes \cdots \otimes v_+ \otimes \cdots \otimes v_- \in V^{\otimes|s_{\sigma_{\beta}}(\hat{\beta})|}=\overline{CKh}(\hat{\beta}_{\sigma_{\beta}}), \] where $v_+$ corresponds to the element in the vector space $V$ associated to the marked circle of $\sigma_{\beta}$. Note that $\widetilde{\psi}'(\beta)$ lives in quantum grading $\Sl(\hat{\beta}) + 1$.

\subsection{Functoriality and properties of $\psi$.} 
Using the map on Khovanov homology induced by a cobordism between a pair of links, Plamenevskaya proved the following useful result for computing $\psi$. 

\begin{thm}{\cite[Theorem 4]{Pla06}} Suppose that the transverse link $\hat{\beta}^-$ represented by the closure of the braid $\beta^-$ is obtained from another transverse link $\hat{\beta}$, also represented by a closed braid,  by resolving a positive crossing (note that it has to be the 0-resolution). Let $S$ be the resolution cobordism, and $f_S: Kh(\hat{\beta}) \rightarrow Kh(\hat{\beta}^-)$ be the associated map on homology, then 
\[f_S(\psi(\beta)) = \pm \psi(\beta^-). \] 
\end{thm}

A consequence of this is that if $\psi(\beta) =0$ then $\psi(\beta^-)=0$. Similarly, suppose that $\hat{\beta}^+$ is obtained from $\hat{\beta}$ by resolving a negative crossing, then $\psi(\beta) \not=0$ implies that $\psi(\beta^+) \not=0$. When we use this property in our computations, we will often cite it as ``functoriality".  Furthermore, this property holds for both the unreduced and reduced versions of the transverse element in the corresponding versions of Khovanov homology.

\subsection{Skein exact sequence} 
Let $D$ be a link diagram and let $D_0$ and $D_1$ be link diagrams differing locally in the 0-resolution and the 1-resolution, respectively, at a negative crossing $c$ of $D$. 
Let $D_1$ inherit the orientation from $D$ and let $u = n_-(D_0) - n_-(D)$ be the difference in the number of negative crossings of the two diagrams, where we pick an orientation on $D_0$. Consider the short exact sequence given by the following maps. 
\[\alpha: CKh^{i}_{j+1}(D_1) \rightarrow CKh_j^i(D) \text{ and } \gamma: CKh^i_j(D) \rightarrow CKh^{i-u}_{j-3u-1}(D_0), \] where $\alpha$ is induced by inclusion, and $\gamma$ is induced by the quotient map.  

 We have the induced long exact sequence below \cite{watson2007knots}, also called the ``skein exact sequence." See \cite{Ras05} for an alternate formulation using the oriented skein relation for the Jones polynomial.
\begin{equation} \label{eq:ssesn} \cdots \rightarrow Kh^{i}_{j+1}(D_1) \stackrel{\alpha}{\rightarrow} Kh^i_{j}(D) \stackrel{\gamma}{\rightarrow} Kh^{i-u}_{j-3u-1}(D_0) \rightarrow Kh^{i+1}_{j+1}(D_1) \rightarrow \cdots.  \end{equation}

For a chosen positive crossing $c$ and with $u = n_-(D_1) - n_-(D)$, we have instead
\begin{equation} \label{eq:ssesp} \cdots \rightarrow Kh^{i-u-1}_{j-3u-2}(D_1) \stackrel{\alpha}{\rightarrow} Kh^i_{j}(D) \stackrel{\gamma}{\rightarrow} Kh^{i}_{j-1}(D_0) \rightarrow Kh^{i-u}_{j-3u-2}(D_1) \rightarrow \cdots.  \end{equation}

These grading shifts can be understood by first considering the shifts of the maps in the exact sequences \emph{before} incorporating the final shifts in $i$ of $-n_{-}$ and in $j$ of $n_{+}-2n_{-}$, and then incorporating those final shifts carefully, keeping in mind that the number of positive and negative crossings in the different diagrams is not the same. Note that the same long exact sequence will hold for reduced Khovanov homology, and over different coefficients.

\subsection{Quasipositivity, right-veeringness, and the fractional Dehn twist coefficient}\label{subsec:defnsofRVQPFDTC}

\begin{defn} A \emph{quasipositive} $n$-braid is a braid that can be expressed as a product of conjugates of the standard positive Artin generators $\sigma_{1}, \ldots, \sigma_{n-1}$. 
\end{defn}

A link is called quasipositive if it is the closure of a quasipositive braid. One reason that quasipositive knots are of interest is that their slice genus can be computed from one of their quasipositive braid representatives \cite{rudolph1993quasipositivity}. While there are obstructions to quasipositivity, there are no known algorithms for determining whether a given link is quasipositive or not. 

We now define the concept of a ``right-veering" braid. To do this we consider the action of the $n$-braid monodromy on the disk $D_{n}$ with $n$ punctures. (Recall that the braid group $B_{n}$ is naturally isomorphic to the mapping class group of $D_{n}$, see for instance \cite{birman_brendle_braids}.) We call an arc in $D_{n}$ starting at a point on $\partial D_{n}$ and ending at a puncture while avoiding all other punctures simply an ``arc on $D_{n}$".  We say that an arc $\eta$ is ``to the right" of an arc $\gamma$ in $D_{n}$ if, after pulling tight to eliminate non-essential intersections, $\eta$ and $\gamma$ originate from the same point on $\partial D_{n}$, and the pair of tangent vectors $(\dot{\eta}, \dot{\gamma})$ at their initial point induces the original orientation on the disk.  See \cite{baldwin2015categorified} and \cite{plamenevskaya2015transverse} for more details. With this terminology in place, we have:

\begin{defn} An $n$-braid is \emph{right-veering} if, under the action of the braid, every arc on $D_{n}$ is sent either to an arc isotopic to itself or to an arc to the right of itself. 
\end{defn}

All quasipositive braids are right-veering, but not all right-veering braids are quasipositive. Recall from the introduction that detecting braids that are right-veering but not quasipositive is generally of interest - see \cite{honda2008right}. 

Finally, we discuss the fractional Dehn twist coefficient, which we will often abbreviate from here on out as the FDTC. If $h$ is any element of the mapping class group of a surface with one boundary component, we denote its FDTC by $\tau(h)$. It roughly measures the amount of twisting effected by the mapping class about the boundary component of the surface. The concept first appeared (though in quite different language) in the work of Gabai and Oertel in \cite{gabai_oertel}. We give here a non-classical definition for braids involving a left order on the braid group as it requires little background. There are several other more geometrically flavored definitions that generalize easily beyond braids to more general mapping class groups. For instance, one way to define the FDTC involves using lifts of the braid to the universal cover of the punctured disk to define a map $\Theta: B_{n} \to \widetilde{Homeo^{+}(S^{1})}$; the FDTC is then defined to be the translation number of $\Theta$.  For a more thorough discussion of this and alternate definitions, see \cite{malyutin2004writhe},  \cite{ito_kawamuro_FDTC}, and \cite{plamenevskaya2015transverse}. 

First, a $\sigma_{i}$-positive $n$-braid is one that, for some $i$ such that $1 \leq i < n$, can be written with no $\sigma_{j}^{\pm 1}$'s for $j < i$ and only positive powers of $\sigma_{i}$. We say that a braid $\beta \in B_{n}$ is Dehornoy positive, that is, $\beta > 1$, if it can be written as a $\sigma_{i}$-positive word. Dehornoy proved in \cite{dehornoy} that this can be used to define a total left-order on the braid group (an order on all of the elements of the braid group that is invariant by multiplication on the left) via the following: we say $\alpha < \beta$ if $\alpha^{-1}\beta > 1$. This order is often called the Dehornoy order on the braid group.

The element $(\sigma_{1} \cdots \sigma_{n-1})^{n}$ is referred to as the full twist in the braid group $B_{n}$, and is denoted by $\Delta^{2}$. The existence of the Dehornoy order on the braid group implies that for every braid $\beta \in B_{n}$, there is a unique integer $m$ such that $\Delta^{2m} \leq \beta < \Delta^{2m+2}$. We denote $m$ as $\lfloor \beta \rfloor$. Malyutin observed in \cite{malyutin2004writhe} that:

\begin{defn}
The fractional Dehn twist coefficient is, for each $\beta \in B_{n}$: $$\tau(\beta) = \displaystyle\lim_{k \to \infty} \frac{\lfloor \beta^{k} \rfloor}{k}.$$
\end{defn}
The following proposition summarizes some basic properties of the FDTC:

\begin{prop}\label{FDTCprops} [See for instance \cite{hedden_mark}.] For any two braids $\alpha$, $\beta$ in $B_{n}$, we have:
\begin{itemize}
\item $|\tau(\alpha\beta) - \tau(\alpha) - \tau(\beta)| \leq 1$.
\item $\tau(\alpha^{-1}\beta\alpha) = \tau(\beta)$.
\item $\tau(\beta^{n}) = n\tau(\beta)$.
\item $\tau(\Delta^{2}\beta) = 1 + \tau(\beta)$. 
\end{itemize}
\end{prop} 

We also have the following result due to Malyutin.

\begin{prop}\label{FDTCestimate} [Malyutin, \cite{malyutin2004writhe}, Lemma 5.4 and Proposition 13.1]  If a braid $\beta \in B_{n}$ is represented by a word that contains $r$ occurrences of $\sigma_{i}$ and $s$ occurrences of $\sigma_{i}^{-1}$ for some $i \in \{1, \ldots, n-1\}$, then $-s \leq \tau(\beta) \leq r$. In particular, if a braid word $\beta \in B_{n}$ is $\sigma_{i}$-free (meaning: it contains no $\sigma_{i}$ or $\sigma_{i}^{-1}$) for some $i \in \{1,\ldots,n-1\}$, then $\tau(\beta)$ is $0$.
\end{prop}

This follows immediately:

\begin{prop}\label{examples} Suppose an $n$-braid $\beta_{n,k}$ has a word of the form $\Delta^{2n}(\alpha)^{k}$ for $\alpha$ any $\sigma_{i}$-free word, $n$ and $k$ integers. Then $\tau(\beta_{n,k})$ is $n$.
\end{prop}

\section{3-braids} \label{sec:3-braids}

We first determine $\psi, \psi'$ for closed 3-braids of the form $\widehat{\triangle^2\sigma_1\sigma_2^{-k}}$ for $k>0$. The reader may skip to Section \ref{subsec:pRVP} to see how these braids come up in the proof of Theorem \ref{3braidFDTC}.

\subsection{The transverse element for the braid $\widehat{\triangle^2\sigma_1\sigma_2^{-k}}$}
\begin{thm}\label{FDTCfamily} Working in the $3$-braid setting, for $k \in \mathbb{Z}$, $k \geq 0$,
 $$\psi(\Delta^{2}\sigma_{1}\sigma_{2}^{-k}) \neq 0$$
when computed over $\mathbb{Q}$, $\mathbb{Z}$, and $\mathbb{Z}/2\mathbb{Z}$ coefficients, and 
 $$\psi'(\Delta^{2}\sigma_{1}\sigma_{2}^{-k}) \neq 0.$$
\end{thm}
Recall that $\psi'$ is the reduced version of the transverse invariant briefly defined in subsection \ref{ss.reduced}. For the definitions, details, and background of many of the notions used in this proof, see the references cited.

\begin{proof} 

Notice first that for $k$ odd, the $3$-braid $\Delta^{2}\sigma_{1}\sigma_{2}^{-k}$ closes to a knot rather than to a link. In \cite{baldwin2008heegaard}, Baldwin showed that the family of $3$-braids
$$\Delta^{2d}\sigma_{1}\sigma_{2}^{-a_{1}}\sigma_{1}\sigma_{2}^{-a_{2}}\cdots\sigma_{1}\sigma_{2}^{-a_{n}}$$
where the $a_{i} \geq 0$ and some $a_{j} \neq 0$ is quasi-alternating if and only if $d \in \{-1,0,1\}$. By work of Manolescu and Ozsv\'ath in \cite{manolescu2007khovanov}, quasi-alternating links are Khovanov homologically $\sigma$-thin. This means that the reduced Khovanov homology over $\mathbb{Z}$ takes a particularly simple form: supported on only one diagonal $j-2i$ grading where $j-2i = \sigma$ the signature of the link \footnote{There are two convention discrepancies between this definition and the cited paper; we are using here the conventions that seem to now be in most common use. In \cite{manolescu2007khovanov}, they consider the grading $j'-i$ instead of $j-2i$ where $j'=\frac{j}{2}$. Their theorem as stated in the paper is that the reduced Khovanov homology over $\mathbb{Z}$ is supported only in grading $j'-i = -\frac{\sigma}{2}$ for quasi-alternating links, or in our grading notation, $j-2i = -\sigma$. The sign discrepancy is explained by the fact that they take the opposite sign convention for the signature as we do: we take as our convention that positive knots have positive signature.}\footnote{We note that in this proof, the symbol $\sigma$ is used to denote only the signature of a link and should not be confused with Kauffman states or braid generators.}. 

As a consequence of the long exact sequence established by Asaeda and Przytycki in \cite{asaeda_przytycki}, Lowrance observed that \cite[Corollary 2.3]{lowrance_khovanov_width} reduced Khovanov homology over $\mathbb{Z}$ has support in grading $j-2i = \sigma$ if and only if Khovanov homology over $\mathbb{Z}$ has support in gradings $j-2i = \sigma+1$ and $j-2i = \sigma-1$. This implies that the Khovanov homology over $\mathbb{Q}$ also only has support in the same gradings. Recall also that Rasmussen's s-invariant \cite{rasmussen2010khovanov} is defined to be the maximum $j$-grading minus one (inherited from the Khovanov complex over $\mathbb{Q}$) of the element in the Lee complex that contributes to Lee homology \cite{lee_khovanov}. Lee homology for knots is particularly simple, and is only supported in $i$-grading $0$ \cite{lee_khovanov}. Hence for Khovanov $\sigma$-thin links, Rasmussen's s-invariant is defined to be the signature $\sigma$. 

Next, notice that for these knots, $\Sl = \sigma - 1$ \cite[Remark 7.6]{baldwin2010khovanov}.  Thus $\Sl = s-1$. Then work of Baldwin and Plamenevskaya \cite[Theorem 1.2]{baldwin2010khovanov} \footnote{This theorem as stated is for reduced Khovanov homology over $\mathbb{Z}/2\mathbb{Z}$ coefficients. However, notice that the proof also explicitly covers the case for Khovanov homology over $\mathbb{Q}$ coefficients and $\mathbb{Z}/2\mathbb{Z}$ coefficients.} implies both that $\psi' \neq 0$  and that $\psi \neq 0$ in Khovanov homology over $\mathbb{Z}/2\mathbb{Z}$ coefficients and $\mathbb{Q}$. This last fact implies that $\psi \neq 0$ in Khovanov homology over $\mathbb{Z}$ coefficients as well. 

Finally, suppose that $k$ is even. Then $\psi(\Delta^{2}\sigma_{1}\sigma_{2}^{-k-1}) \neq 0$ and $\psi'(\Delta^{2}\sigma_{1}\sigma_{2}^{-k-1}) \neq 0$ by what we just proved. By functoriality, $\psi(\Delta^{2}\sigma_{1}\sigma_{2}^{-k}) \neq 0$ and $\psi'(\Delta^{2}\sigma_{1}\sigma_{2}^{-k}) \neq 0$.
\end{proof}

\subsection{Proof of Theorem \ref{3braidFDTC}} \label{subsec:pRVP} We restate the theorem for convenience:
\begin{restate2} Suppose $K$ is a transverse knot that has a $3$-braid representative $\beta$ with fractional Dehn twist coefficient $\tau(\beta) > 1$. Then $\psi(K) \neq 0$ when computed over $\mathbb{Q}$, $\mathbb{Z}$, and $\mathbb{Z}/2\mathbb{Z}$ coefficients, and $\psi'(K) \neq 0$.
\end{restate2}

\begin{proof}
We will write this proof only for $\psi$; it is identical for $\psi'$. According to Murasugi's classification of 3-braids \cite{murasugi1974closed}, every $\sigma \in B_{3}$ comes in the following types up to conjugation:
\begin{enumerate}[a)]
\item $\Delta^{2d}\sigma_{1}\sigma_{2}^{-a_{1}}\sigma_{1}\sigma_{2}^{-a_{2}}\cdots\sigma_{1}\sigma_{2}^{-a_{n}}$ where $a_{i} \geq 0$ for all $i$ and some $a_{i} > 0$,
\item $\Delta^{2d}\sigma_{2}^{m}$ where $m \in \mathbb{Z}$, and
\item $\Delta^{2d}\sigma_{1}^{m}\sigma_{2}^{-1}$ where $m=-1, -2, -3$,
\end{enumerate}
where $d$ can take on any integer value.
Recall that $\psi$ is invariant under conjugation, so whichever of these conjugacy classes $\sigma$ belongs to determines $\psi(\sigma)$. The FDTC is also invariant under conjugation, and hence whichever of these conjugacy classes $\sigma$ belongs to determines its FDTC.

All braids in classes (a) and (b) have FDTC $d$, since $$\tau(\sigma_{2}^{-a_{1}}\sigma_{1}\sigma_{2}^{-a_{2}}\cdots\sigma_{1}\sigma_{2}^{-a_{n}}) = 0$$ and $\tau(\sigma_{2}^{m}) = 0$ by Proposition \ref{FDTCestimate}, and for any braid $\beta$, $\tau(\Delta^{2}\beta) = 1 + \tau(\beta)$. All braids in class (c) have FDTC less than or equal to $d$, since by Proposition \ref{FDTCestimate}, $\tau(\sigma_{1}^{m}\sigma_{2}^{-1}) \leq 0$ for negative values of $m$.

Hence for each of the classes, we need only consider $d > 1$. Since by Theorem \ref{FDTCfamily} we know that the model braid $\psi(\Delta^{2}\sigma_{1}\sigma_{2}^{-k}) \neq 0$ for all positive $k$ then every other braid in (a) and (b) with $d > 1$ has $\psi \neq 0$ by functoriality. Indeed: by making $k$ possibly quite large, we can achieve every other braid in (a) and (b) with $d > 1$ by inserting positive crossings. 

Finally, a straightforward manipulation of the braid words yields that the braids in (c) with $d > 1$ are all quasipositive. Hence for the braids in (c) with $d > 1$, $\psi \neq 0$ (\cite{Pla06}). 

\end{proof}

\section{General stability}\label{sec:genstability}
We prove Theorem \ref{thm:stab} in this section. Note in general that we have the following bounds on $Kh^i_j$. If $D$ is a diagram of a link $K$, then $Kh^i_j(D) = 0$ if $i$ or $j$ are outside of the following bounds:
\begin{align*}
-n_-(D) &\leq i \leq n_+(D) \\
n_+(D)-2n_-(D)-|s_0(D)| &\leq j \leq |s_1(D)| + 2n_+(D) - n_-(D). 
\end{align*}

\subsection{Negative sub-full twists }

Let $\beta$ be a braid of strand number $b$ and let $2\leq a  < b$. Let $k$ be a positive integer, and write $k = (a-1) \ell + r$, so $r = k \mod (a-1)$. We consider the closed braid $D^k$ obtained by adding to $\beta$ the following braid
\[ \alpha'_-= (\sigma_{i}^{-1}\sigma_{i+1}^{-1}\cdots\sigma_{i+a-2}^{-1})^{\ell}(\sigma_{i}^{-1}\sigma_{i+1}^{-1} \cdots \sigma_{i+r-1}^{-1})\] of strand number $a$, with $1\leq i \leq b-a+1$, and then taking the closure.  

\begin{figure}[ht]
\begin{center}
\def \svgwidth{.2\columnwidth}
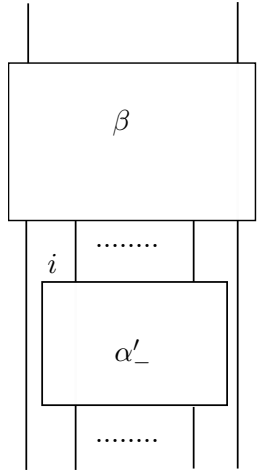
\caption{\label{fig:schema} The braid closing to $D^k$. }
\end{center}
\end{figure}

We denote by $D_0^k$ and $D_1^k$ the link diagrams obtained by taking the 0-resolution and the 1-resolution, respectively, at the crossing $\sigma_{i+r-1}^{-1}$. 

\begin{figure}[ht]
\def \svgwidth{.5\columnwidth}
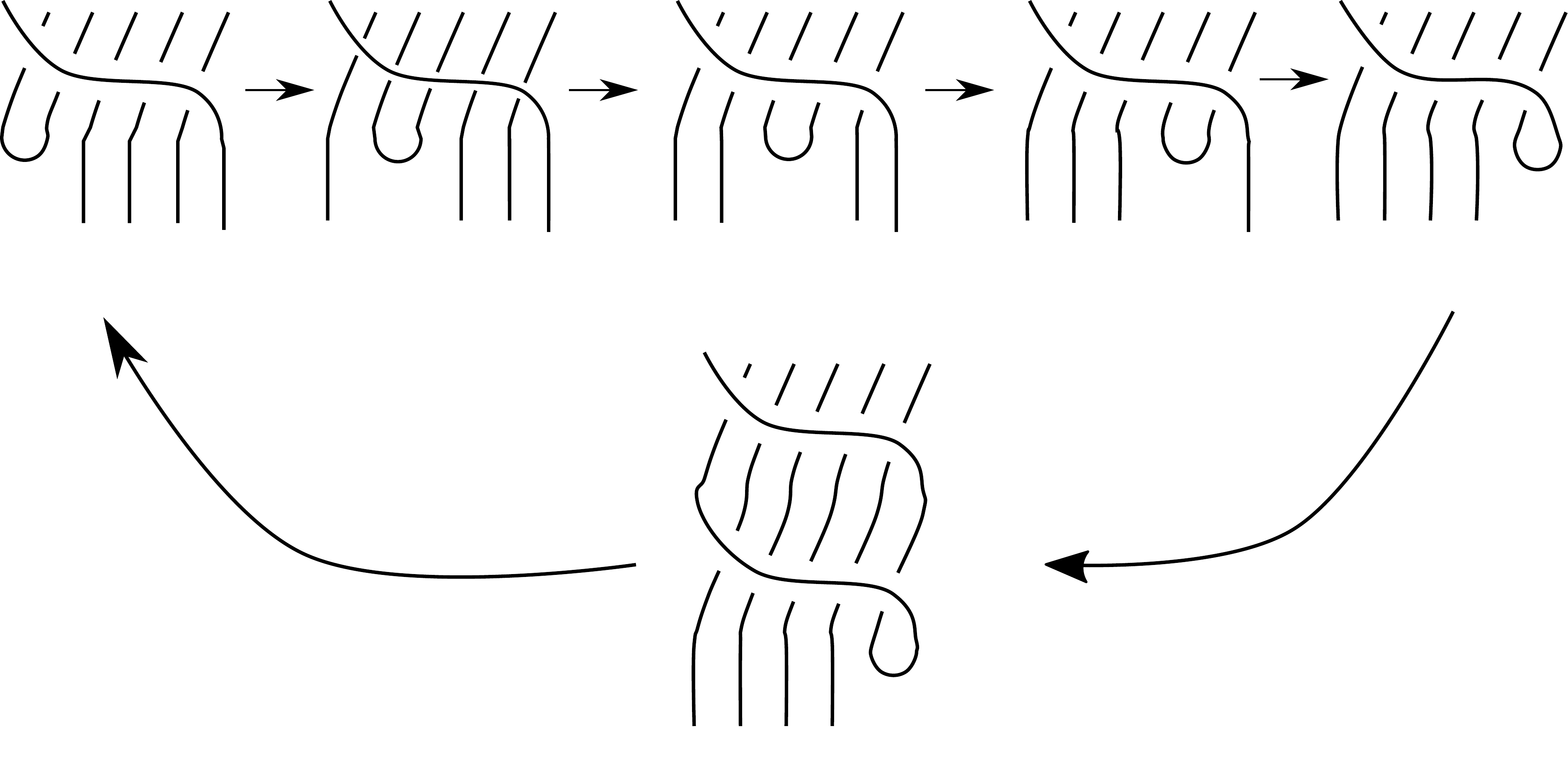
\caption{\label{fig:count} An example where $a = 6$ and $r = 1$ as we isotope the cap through copies of $\sigma_i^{-1} \sigma_{i+1}^{-1}\cdots \sigma_{i+a-2}^{-1}$. }
\end{figure} 
\begin{lem} \label{l.exceed}
Let $n_+'(D_0^k)$ be the number of positive crossings of $D_0^k$ in the subset $\alpha'_-$.
Then
\begin{equation} \label{eq:est} n_+'(D_0^k) -  n_+'(\widetilde{D_0^k}) \geq \ell,   \end{equation}
where $\widetilde{D_0^k}$ is the diagram isotopic to $D_0^k$, obtained by isotoping the cap through the rest of the braid $\alpha_-'$ resulting from choosing the $0$-resolution at the crossing $\sigma^{-1}_{i+r-1}$ in $D^k$ to get $D_0^k$. See Figure \ref{fig:count} for an example. 
\end{lem}
\begin{proof}
Denote the strands of the braid $D^k$ by $S_1, \ldots, S_b$. We follow the isotopy of the cap resulting from choosing the $0$-resolution at the crossing $\sigma_{i+r-1}^{-1}$ through the $\ell$ copies of $(\sigma_{i}^{-1}\sigma_{i+1}^{-1}\cdots \sigma_{i+a-2}^{-1})$ as shown above in Figure \ref{fig:count} in an example where $r=1$. Initially, the cap joins the strands $S_{i}$ and $S_{i+r}$. Regardless of the choice of orientation on $D_0^k$, we end up decreasing the number of positive crossings by one for each set of  $(\sigma_{i}^{-1}\sigma_{i+1}^{-1} \cdots \sigma_{i+a-2}^{-1})$ in $\alpha'$ through the isotopy. See also Figure \ref{f.orientation}.

\begin{figure}
\begin{center}
\def \svgwidth{.2\columnwidth}
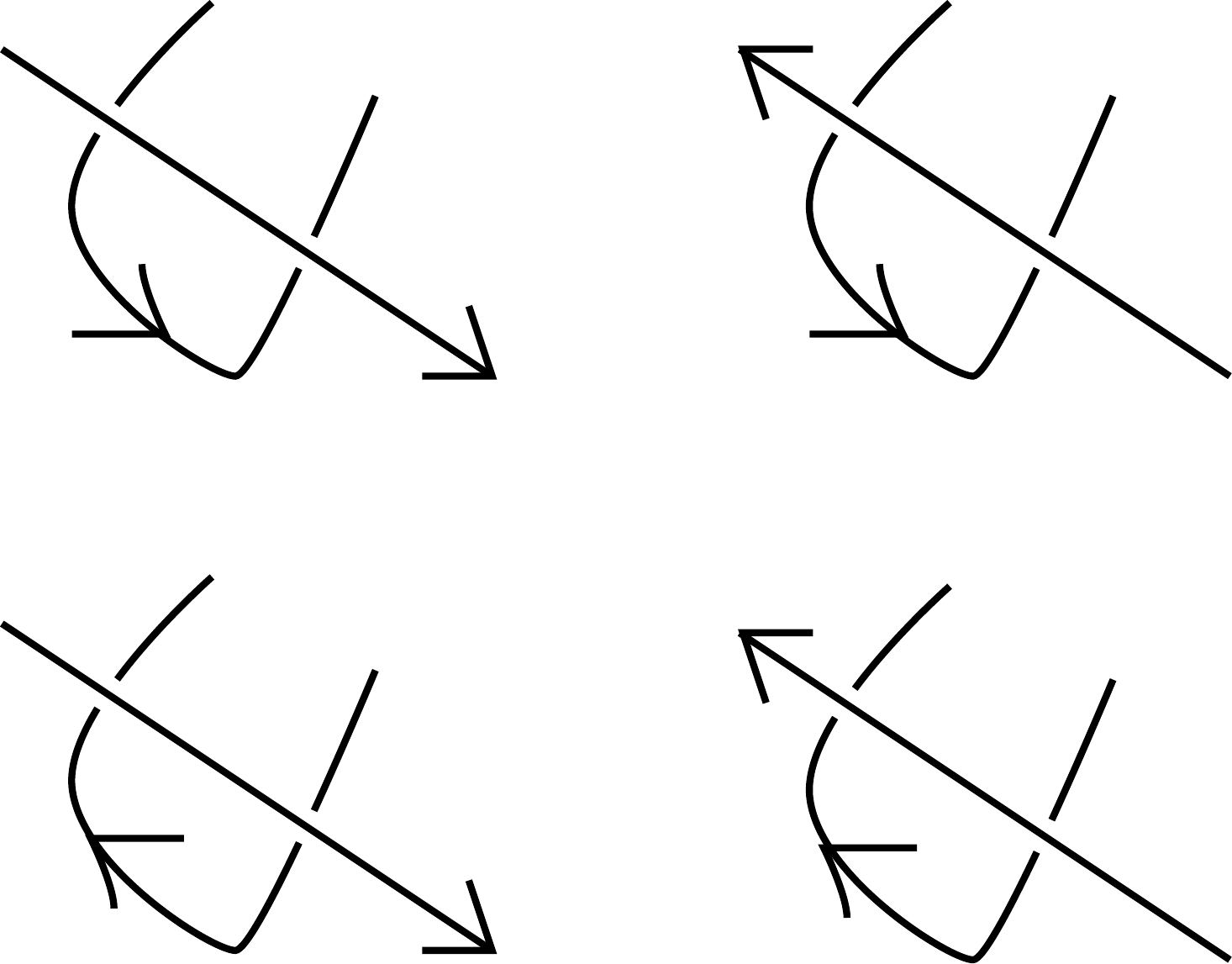
\caption{\label{f.orientation}We see the cap and the over-strand under which it passes through isotopy. Regardless of the choice of orientation, the number of positive crossings is decreased by 1 through each passing under a strand.}
\end{center}
\end{figure} 

\end{proof} 
We use the long exact sequence in Khovanov homology at a distinguished negative crossing \eqref{eq:ssesn}. Let $u = n_-(D_0^k) - n_-(D^k)$, the long exact sequence takes the form: 
\[\cdots \rightarrow  Kh^{i-1-u}_{j-3u-1}(D_0^k) \rightarrow Kh^{i}_{j+1}(D^k_1)\rightarrow Kh^i_j(D^k) \rightarrow Kh^{i-u}_{j-3u-1}(D_0^k) \rightarrow Kh^{i+1}_{j+1}(D_1^k) \rightarrow \cdots \] 
Then we show that with $k$ large enough, $Kh^{-1-u}_*(D_0^k)= 0$ and $Kh^{-u}_*(D_0^k)=0$ by showing 
$-u -1> n_+(\widetilde{D_0^k})$. This implies 
\[Kh^0_{j+1}(D_1^k) \stackrel{\alpha}{\cong} Kh^0_{j}(D^k),  \] with 
\[ \alpha(\widetilde{\psi}(D_1^k)) = \widetilde{\psi}(D^k) .\]

\begin{lem} \label{thm:ltrivial} Given $D^k = \widehat{\beta \alpha'_-}$, there exists a fixed number $N> 0 $ such that for all $k\geq N$, the homology groups
\[Kh^{-1-u}_*(D_0^k) \text{ and } Kh^{-u}_*(D_0^k)  \] are both trivial.
\end{lem} 
\begin{proof}
We need to show 
\begin{align*}
n_-(D^k) - n_-(D_0^k)-1 &> n_+(\widetilde{D_0^k}).
\intertext{Let $n'_-(D^k)$ be the number of negative crossings of $D^k$ in the subset $\alpha'_-$ and  $n^{\beta}_{\pm}(D^k) = n_{\pm}(D_k) - n'_{\pm}(D^k)$ be the number of positive/negative crossings of $D^k$ in the subset $\beta$. We rewrite the inequality as}
n'_-(D^k) - n'_-(D_0^k) + n^{\beta}_-(D^k) - n^{\beta}_-(D_0^k)-1 &> n'_+(\widetilde{D_0^k}) + n_+^{\beta}(\widetilde{D_0^k}). \\
\intertext{The following inequality obtained from rewriting $n_-'(D^k) - n_-'(D_0^k) = n'_+(D_0^k)+1$ and using Lemma \ref{l.exceed} implies the desired inequality above.}
n'_+(D_0^k) + n^{\beta}_-(D^k) - n^{\beta}_-(D_0^k) &> n'_+(D_0^k)- \ell + n_+^{\beta}(D^k) + n^{\beta}_-(D^k) - n^{\beta}_-(D_0^k).\\
\intertext{This simplifies to}
0 &> - \ell + n_+^{\beta}(D^k).
\end{align*} 
We can certainly make the last inequality true by making $k$ large enough so that $\ell > n_+(D^k)$, since $n_+^{\beta}(D^k)$ is constant. 
\end{proof} 

\subsection{Proof of Theorem \ref{thm:stab} for adding negative sub-full twists}
\begin{proof}
Let $L^m_{\pm} = \widehat{\beta (\alpha^{\pm})^m}$ as in the statement of the theorem. 
Choose large enough $m$ so that 
\[ \psi(L^{m}_-) = \psi(\beta(\alpha^-)^m\sigma_i^{-1}) =\psi(\beta(\alpha^-)^m\sigma_i^{-1} \sigma_{i+1}^{-1}) = \cdots = \psi(\beta(\alpha^-)^m\sigma_i^{-1}\sigma_{i+1}^{-1}\cdots \sigma^{-1}_{i+a-2}) = \psi(L^{m+1}_-)\] 
 by Lemma \ref{thm:ltrivial}. The conclusion of the theorem follows. 
\end{proof} 

\subsection{Positive sub-full twists}

\noindent This proof is analogous to the one for negative sub-full twists; the primary difference is that we use the long exact sequence at a distinguished positive crossing \eqref{eq:ssesp} with $u = n_-(D_1^k) - n_-(D^k)$, and we show that $-u < -n_{-}(\widetilde{D_{1}^{k}})$ instead of the bound on the homological degree on the other side as the isotopy that simplifies $D_{1}^{k}$ to $\widetilde{D_{1}^{k}}$ will reduce the number of negative crossings.  We use the same notation as before for indicating the positive/negative crossings in different regions of the braid.

We consider the closed braid $D^k$ obtained from adding to $\beta$ the following braid
\[ \alpha'_+ = (\sigma_{i}\sigma_{i+1}\cdots\sigma_{i+a-2})^{\ell}(\sigma_{i}\cdots \sigma_{i+r-1}), \] of strand number $2\leq a< b$, with $k = (a-1) \ell+r$, and then taking the closure. Let $D_1^k$ be the link diagram obtained by choosing the 1-resolution at the crossing $\sigma_{i+r-1}$. We obtain the analogous statement $n_-'(D_1^k) -  n_-'(\widetilde{D^k_1}) \geq \ell$
 to Lemma \ref{l.exceed} by replacing $n'_+(D^k_0)$ with $n'_-(D^k_1)$ and $n_+(\widetilde{D^k_0})$ with $n_-(\widetilde{D^k_1})$. The argument is similar except that the cap from choosing the $1$-resolution at $\sigma_{i+r-1}$  is now over the other braid strands. The inequality follows that
$$ n_-(D^k) - n_-(D_1^k)< -n_-(\widetilde{D_1^k}), $$ whenever $\ell > n_-(D^k)$. 

\subsection{Stability for the reduced version}

Note that we have the same bounds for reduced Khovanov homology with $\Z/2\Z$ coefficients on the homological grading: $\overline{Kh}^{i}_{j}(D) = 0$ if $i$ does not satisfy $-n_-(D)\leq i \leq n_+(D)$. Thus the same proof as above goes through to show stability for $\psi'$ under adding positive/negative sub-full twists using the long exact sequence for the reduced version.  

\section{Applications and examples}\label{sec:appsofgenstability}

In this section we apply the collection of tools we now have to determine the behavior of $\psi$ and $\psi'$ for a few families of closed braids, and draw conclusions about their quasipositivity and right-veeringness. For $\psi'$, we use Baldwin's program together with the stability behavior of $\psi'$ proved in Section \ref{sec:genstability}; for $\psi$, we will use by-hand computation with stability. For instance, one can determine the behavior of $\psi'$ and $\psi$ for the $3$-braid family from Theorem \ref{FDTCfamily} in this way, as it is possible to check that $\psi'$ does not vanish for $\Delta^{2}\sigma_{1}\sigma_{2}^{-8}$ using Baldwin's program, and $\psi\not=0$ by hand. In cases where the bound in Section \ref{sec:genstability} would require checking an example with too many crossings for Baldwin's program to handle, it is sometimes still possible to use the same general approach to get more precise information, as we do in subsection \ref{subsec:4braidex} for $\psi'$ for a family of 4-braids.

\subsection{A collection of examples}\label{subsec:appsofgenstability}
The first four columns of Table \ref{table:qprv} denote the number of strands of the braid, the word template for the braid family that we consider, the behavior of $\psi$ and $\psi'$ for these braids that we are able to determine\footnote{While we have no examples where the behaviors of $\psi$ and $\psi'$ differ, we know of no mathematical reason why their behaviors should \emph{always} match.}, and the methods used to obtain these results:  ``Prog." stands for Baldwin's program for $\psi'$ and ``Comp." stands for a by-hand computation for $\psi$. Wherever we claim that $\psi$ dies due to a by-hand computation, we provide the element that kills it in subsection \ref{subsec:killpsi}. The fifth column gives the writhe of the braid, and the sixth and seventh columns determine whether the braid is quasipositive and/or right-veering, if possible, along with the method used. We have:

\begin{itemize}
\item Braid families that are right-veering but not quasipositive (the first six).
\item Braid families that are not quasipositive and have positive writhes (the last three).
\end{itemize}

\begin{table}[H]
\begin{center}

\resizebox{1\textwidth}{!}{%
\begin{tabular}{||c | c | c | p{22mm} | c | p{22mm} | p{19mm}||} 
 \hline
 $n$ & Braid in $B_{n}$ & $\psi, \psi'$ & Method & Writhe & Quasipositive & Right-veering \\ 
 \hline\hline
3 & $\Delta^{2}\sigma_{2}^{-k}, k > 4$ & $\psi, \psi'=0$ & Prog,./Comp., functoriality & $6-k$ & No, $k > 6$, writhe & Yes, FDTC  \\ 
 \hline
 3 & $\Delta^{2}\sigma_{1}\sigma_{2}^{-k}, k \in \mathbb{N}$ & $\psi, \psi' \neq 0$ & See Thm \ref{FDTCfamily}  & $7-k$ & No, $k > 7$, writhe & Yes, $\psi/\psi'$ or FDTC \\
 \hline
 4 & $\Delta^{2}\sigma_{2}^{-k}, k \in \mathbb{N}$ & $\psi, \psi' \neq 0$ & Example above, functoriality & $12-k$ & No, $k > 12$, writhe & Yes, $\psi/\psi'$ or FDTC \\
 \hline
 4 & $\Delta^{2}\sigma_{3}^{-k}, k \in \mathbb{N}$ & $\psi' \neq 0$ & See subsection \ref{subsec:4braidex} & $12-k$ & No, $k > 12$, writhe & Yes, $\psi'$ or FDTC \\
 \hline
 4 & $\sigma_{1}\sigma_{2}\sigma_{3}\sigma_{3}\sigma_{2}\sigma_{1}\sigma_{3}^{-k}, k > 2$ & $\psi, \psi' =0$ & Prog./Comp., functoriality & $6-k$ & No, $k > 6$, writhe & Yes, FDTC \\
 \hline
 4 & $\Delta^{2}(\sigma_{2}\sigma_{3})^{-k}, k > 5$ & $\psi, \psi' =0$ & Example above\footnotemark, functoriality & $12-2k$ & No, $k > 6$, writhe & Yes, FDTC \\
 \hline
 4 & $(\sigma_{1})^2\sigma_{2}^{-1}\sigma_{3}\sigma_{2}^{-1}\sigma_{1}^{-1}\sigma_{2}(\sigma_{3})^2(\sigma_{2}\sigma_{3})^{k}, k \in \mathbb{N}$ & $\psi, \psi' =0$ & Prog./Comp., stability & $3+2k$ & No, all $k$, $\psi/\psi'$ & ? \\
 \hline
 5 & $\sigma_{1}\sigma_{2}^{-1}\sigma_{3}\sigma_{4}^{-1}\sigma_{2}^{-1}\sigma_{1}^{-1}(\sigma_{2})^{2}\sigma_{3}(\sigma_{4})^2(\sigma_{2}\sigma_{3})^{k}, k \in \mathbb{N}$ & $\psi, \psi' =0$ & Prog./Comp., stability & $3+2k$ & No, all $k$, $\psi/\psi'$ & ? \\
\hline
6 &  $\sigma_{4}\sigma_{1}\sigma_{2}\sigma_{4}\sigma_{5}^{-1}\sigma_{4}^{-1}\sigma_{3}\sigma_{5}\sigma_{1}^{-1}\sigma_{2}(\sigma_{2}\sigma_{3})^{k}, k \in \mathbb{N}$ & $\psi, \psi' =0$ & Prog./Comp., stability & $4+2k$ & No, all $k$, $\psi/\psi'$ & ? \\
\hline
\end{tabular}
}
\end{center}
\caption{Calculations of $\psi$ and $\psi'$ for various braids on $n$ strands, together with their calculation method as well as other properties of these braids.}\label{table:qprv}
\end{table}
\vspace*{3mm}

\footnotetext{{Using the fact that $\Delta^{2}(\sigma_{2}\sigma_{3})^{-3} = \sigma_{1}\sigma_{2}\sigma_{3}\sigma_{3}\sigma_{2}\sigma_{1}$}.}

Recall that it is of interest to detect braids that are right-veering but not quasipositive - see subsection \ref{subsec:defnsofRVQPFDTC}. Baldwin and Grigsby proved in \cite{baldwin2015categorified} that if a braid is not right-veering, then it has vanishing $\psi$. Their proof would apply just as well to $\psi'$. Hence if a braid has non-vanishing $\psi$ or $\psi'$, it is guaranteed to be right-veering. Notice also that if one has a braid with negative writhe, then it cannot be quasipositive. Thus $\psi$ or $\psi'$ together with the writhe can be used to detect braids that are right-veering but not quasipositive, as is done in the second through fourth examples in the table. However, it is also possible for $\psi$ and $\psi'$ to vanish for braids that are right-veering but not quasipositive, as can be seen in the first, fifth, and sixth examples in the table. In these cases we were able to determine that the FDTCs for these braid families were greater than or equal to one, which implies that these braids are indeed right-veering \cite{honda2008right}.

In the case where a braid has positive writhe, $\psi$ or $\psi'$ can be of use to detect non-quasipositivity. Indeed, Plamenevskaya proved in \cite{Pla06} that if a braid is quasipositive, then it has non-vanishing $\psi$, and her proof applies equally well to $\psi'$.  The last three examples in the table have arbitrarily large writhes but also have vanishing $\psi$ and $\psi'$, and hence are not quasipositive. We chose these examples as it is not obvious by simply manipulating the braid words that they are not quasipositive; there should be many more such examples.

\subsection{Justification for $\psi=0$ in Table \ref{table:qprv}}\label{subsec:killpsi}
For each braid $\beta$ in Table \ref{table:qprv} where we state that ``comp." is the method for showing $\psi(\beta) = 0$, we justify the claim by giving an explicit element that can be directly verified to kill $\psi$, that is, we give an element $\Phi \in CKh^{-1}_{\Sl(\hat{\beta})} = CKh^{-1}_{\Sl(\hat{\beta})}(\hat{\beta})$ such that $d(\Phi) = \psi$. These braids are (in order of their appearance from top to bottom in Table \ref{table:qprv}):  
\begin{enumerate}
\item $\triangle^2 \sigma_2^{-k}$, $k>4$
\item $\sigma_1\sigma_2\sigma_3\sigma_3\sigma_2\sigma_1\sigma_3^{-k}$, $k>2$
\item  $(\sigma_1)^2\sigma_2^{-1}\sigma_3\sigma_2^{-1}\sigma_1^{-1}\sigma_2(\sigma_3)^2(\sigma_2\sigma_3)^k$, $k\in \mathbb{N}$
\item $\sigma_1\sigma_2^{-1}\sigma_3\sigma_4^{-1}\sigma_2^{-1}\sigma_1^{-1} (\sigma_2)^2\sigma_3(\sigma_4)^2(\sigma_2\sigma_3)^k, k\in \mathbb{N}$ 
\item $\sigma_4\sigma_1\sigma_2\sigma_4\sigma_5^{-1}\sigma_4^{-1}\sigma_3\sigma_5\sigma_1^{-1}\sigma_2(\sigma_2\sigma_3)^k, k \in \mathbb{N}$
\end{enumerate} 
To represent an element in $CKh^{-1}_{\Sl(\hat{\beta})}$ that kills $\psi$, we represent a generator in $CKh$ by first giving its Kauffman state by decorating the braid word $\beta$ with dots. A dot on top of an Artin generator indicates that the $1$-resolution is chosen at the corresponding crossing. Without the dot, the $0$-resolution is chosen. Then, we number the state circles of the Kauffman state and indicate between parentheses which circle is marked with a +, corresponding to $v_+\in V$, in the grading $i=-1$, $j = \Sl(\hat{\beta})$.  Each circle in the state not indicated between parentheses is marked with a $-$, corresponding to $v_-\in V$. If there are two numbers in the parentheses then the corresponding circles are both marked with a $+$. In the figures, a segment between state circles indicates where the crossing is before the Kauffman state is applied. Thickened (or blue) segments indicate that the $1$-resolution is chosen, and thin (or red) segments indicate that the $0$-resolution is chosen.  See the following figure for an example of how to interpret a dotted braid word and how to read off the generator from a Kauffman state on the closed braid where circles are labeled. 

\begin{figure}[ht]
\begin{center} 
\def \svgwidth{.7\columnwidth}
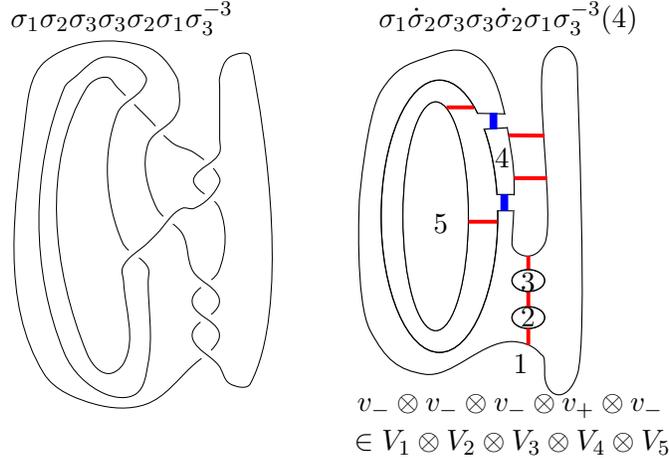
\end{center} 
\caption{On the left: the closure of a 4-braid. On the right: a generator in a Kauffman state on the closed 4-braid. The tensor product $V^{\otimes 5}$ is formed from a Kauffman state on the closure of the braid $\sigma_1\sigma_2\sigma_3\sigma_3\sigma_2\sigma_1\sigma_3^{-3}$, which chooses the 1-resolution at the two crossings corresponding to $\sigma_2$, and the 0-resolution everywhere else. The factors of the tensor product are ordered according to the label on each of the resulting state circles. We see that the circle labeled with 4 is marked with a $+$ and all other circles are labeled with a $-$.}
\end{figure} 
\begin{enumerate} 
\item $\triangle^2 \sigma_2^{-k}$, $k>4$. 
By the braid relations, we get $\triangle^2\sigma_2^{-5} = \sigma_1\sigma_2\sigma_2\sigma_1\sigma_2^{-3}$. If we show $\psi = 0$ for the closure of $\sigma_1\sigma_2\sigma_2\sigma_1\sigma_2^{-3}$, then functoriality would show that $\psi = 0$ for $\widehat{\triangle^2\sigma_2^{-k}}$ for all $k>5$.  We claim 
\begin{align*} \label{3psi1}
&\psi(\sigma_1\sigma_2\sigma_2\sigma_1\sigma_2^{-3})\\
&= d(-\dot{\sigma_1}\sigma_2\sigma_2\dot{\sigma_1}\sigma_2^{-3}(4) + \sigma_1\dot{\sigma_2}\sigma_2\dot{\sigma_1}\sigma_2^{-3}(4) - \sigma_1\dot{\sigma_2}\sigma_2\sigma_1\sigma_2^{-1}\dot{\sigma_2}^{-1}\sigma_2^{-1}(1) + \sigma_1\sigma_2\dot{\sigma_2}\dot{\sigma_1}\sigma_2^{-3}(4) \\ \notag
&+ \sigma_1\dot{\sigma_2}\sigma_2\sigma_1\sigma_2^{-1}\dot{\sigma_2}^{-1}\sigma_2^{-1}(4)-\sigma_1\dot{\sigma_2}\sigma_2\dot{\sigma_1}\sigma_2^{-3}(3)-\sigma_1\sigma_2\dot{\sigma}_2\sigma_1\sigma_2^{-1}\dot{\sigma}_2^{-1}\sigma_2^{-1}(1) + \sigma_1\sigma_2\dot{\sigma}_2\sigma_1\sigma_2^{-1}\dot{\sigma}_2^{-1}\sigma_2^{-1}(4)\\ \notag
&-\sigma_1\sigma_2\dot{\sigma_2}\dot{\sigma_1}\sigma_2^{-3}(3)-\sigma_1\dot{\sigma_2}\dot{\sigma_2}\sigma_1\sigma_2^{-3}(45)+ \sigma_1\dot{\sigma_2}\dot{\sigma_2}\sigma_1\sigma_2^{-3}(35)+ \sigma_1\sigma_2\sigma_2\sigma_1\sigma_2^{-1}\dot{\sigma_2}^{-1}\dot{\sigma_2}^{-1}), 
\end{align*}
where $\psi$ is the image of $d$ on the element indicated on the right hand side of the equation. See Figure \ref{f.3-braid} for the labeling of circles of the states in $CKh^{-1}_{\Sl(\widehat{\sigma_1\sigma_2\sigma_2\sigma_1\sigma_2^{-3})})}$. 
\begin{figure}[ht]
\centering
\scriptsize{
\def \svgwidth{.9\columnwidth}
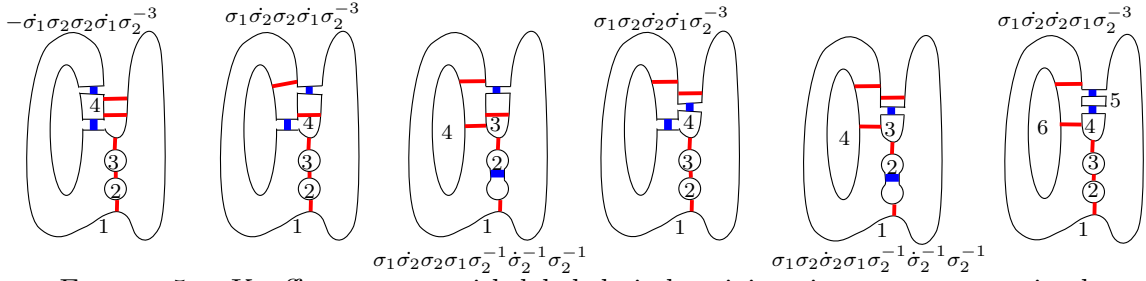}
\caption{\label{f.3-braid} Kauffman states with labeled circles giving rise to generators in the sum killing $\psi$ for the closure of the 3-braid $\sigma_1\sigma_2\sigma_2\sigma_1\sigma_2^{-3}$.   }
\end{figure} 
\item $\sigma_1\sigma_2\sigma_3\sigma_3\sigma_2\sigma_1\sigma_3^{-k}$, $k>2$. Also by functoriality, it suffices to show $\psi$ vanishes for the closure of  $\sigma_1\sigma_2\sigma_3\sigma_3\sigma_2\sigma_1\sigma_3^{-3}$. We claim

\begin{align*}
&\psi(\sigma_1\sigma_2\sigma_3\sigma_3\sigma_2\sigma_1\sigma_3^{-3})\\ 
&= d(\sigma_1\sigma_2\sigma_3\sigma_3\sigma_2\sigma_1 \sigma_3^{-1}\dot{\sigma}_3^{-1}\dot{\sigma}_3^{-1}-\sigma_1\dot{\sigma}_2\sigma_3\sigma_3\dot{\sigma}_2\sigma_1 \sigma_3^{-3}(4)+\sigma_1\sigma_2\dot{\sigma}_3\sigma_3\dot{\sigma}_2\sigma_1 \sigma_3^{-3}(4) \\ 
&-\sigma_1\sigma_2\dot{\sigma}_3\sigma_3\dot{\sigma}_2\sigma_1 \sigma_3^{-3}(5)+\sigma_1\sigma_2\sigma_3\dot{\sigma}_3\dot{\sigma}_2\sigma_1 \sigma_3^{-3}(4)-\sigma_1\sigma_2\sigma_3\dot{\sigma}_3\dot{\sigma}_2\sigma_1 \sigma_3^{-3}(5)-\sigma_1\sigma_2\dot{\sigma}_3\sigma_3\sigma_2\sigma_1 \dot{\sigma}_3^{-1}\sigma_3^{-2}(3)\\
&-\sigma_1\sigma_2\dot{\sigma}_3\sigma_3\sigma_2\sigma_1 \sigma_3^{-1}\dot{\sigma}_3^{-1}\sigma_3^{-1}(2)-\sigma_1\sigma_2\sigma_3\dot{\sigma}_3\sigma_2\sigma_1 \sigma_3^{-1}\dot{\sigma}_3^{-1}\sigma_3^{-1}(2)-\sigma_1\sigma_2\sigma_3\dot{\sigma}_3\sigma_2\sigma_1\dot{\sigma}_3^{-1}\sigma_3^{-2}(3)\\
&-\sigma_1\sigma_2\dot{\sigma}_3\dot{\sigma}_3\sigma_2\sigma_1 \sigma_3^{-3}(45)).
\end{align*}
See Figure \ref{f.4braid1} for the labeling of circles of the states in $CKh^{-1}_{\Sl(\widehat{\sigma_1\sigma_2\sigma_3\sigma_3\sigma_2\sigma_1\sigma_3^{-3}})}$. 

\begin{figure}[ht]
\centering
\scriptsize{
\def \svgwidth{\columnwidth}
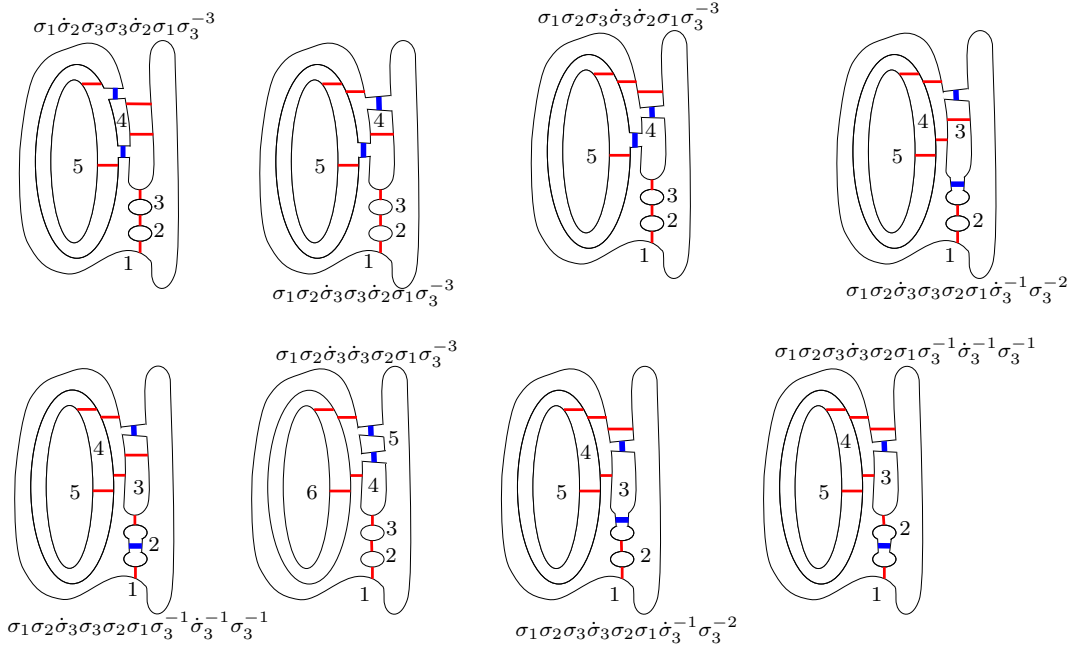}
\caption{\label{f.4braid1} States and numbered circles for the closure of the 4-braid $\sigma_1\sigma_2\sigma_3\sigma_3\sigma_2\sigma_1\sigma_3^{-3}$.}
\end{figure}

\item $(\sigma_1)^2\sigma_2^{-1}\sigma_3\sigma_2^{-1}\sigma_1^{-1}\sigma_2(\sigma_3)^2(\sigma_2\sigma_3)^k$, $k\in \mathbb{N}$. We claim for all $k\in \mathbb{N}$, 

\begin{align*}
&\psi((\sigma_1)^2\sigma_2^{-1}\sigma_3\sigma_2^{-1}\sigma_1^{-1}\sigma_2(\sigma_3)^2(\sigma_2\sigma_3)^k) \\ 
 &= d((\sigma_1)^2\dot{\sigma}_2^{-1}\sigma_3\dot{\sigma}_2^{-1}\sigma_1^{-1}\sigma_2(\sigma_3)^2(\sigma_2\sigma_3)^k-\sigma_1\dot{\sigma}_1\dot{\sigma}_2^{-1}\sigma_3\sigma_2^{-1}\sigma_1^{-1}\sigma_2(\sigma_3)^2(\sigma_2\sigma_3)^k \\ 
 &-\dot{\sigma}_1\sigma_1\dot{\sigma}_2^{-1}\sigma_3\sigma_2^{-1}\sigma_1^{-1}\sigma_2(\sigma_3)^2(\sigma_2\sigma_3)^k).
\end{align*}
All the circles of all these states are marked with a $-$. Note that the same elements would map to $\psi$ regardless of $k$.

\item $\sigma_1\sigma_2^{-1}\sigma_3\sigma_4^{-1}\sigma_2^{-1}\sigma_1^{-1} (\sigma_2)^2\sigma_3(\sigma_4)^2(\sigma_2\sigma_3)^k, k\in \mathbb{N}$. We claim for all $k\in \mathbb{N}$, 
\begin{align*} 
&\psi(\sigma_1\sigma_2^{-1}\sigma_3\sigma_4^{-1}\sigma_2^{-1}\sigma_1^{-1} (\sigma_2)^2\sigma_3(\sigma_4)^2(\sigma_2\sigma_3)^k)\\
&= d(\sigma_1\dot{\sigma}_2^{-1}\sigma_3\dot{\sigma_4}^{-1}\dot{\sigma}_2^{-1}\sigma_1^{-1}\sigma_2^2\sigma_3\sigma_4^2(\sigma_2\sigma_3)^k-\dot{\sigma}_1\dot{\sigma}_2^{-1}\sigma_3\dot{\sigma}_4^{-1}\sigma_2^{-1}\sigma_1^{-1}\sigma_2^2\sigma_3\sigma_4^2(\sigma_2\sigma_3)^k). 
\end{align*}
All the circles of all these states are marked with a $-$.
\item $\sigma_4\sigma_1\sigma_2\sigma_4\sigma_5^{-1}\sigma_4^{-1}\sigma_3\sigma_5\sigma_1^{-1}\sigma_2(\sigma_2\sigma_3)^k, k \in \mathbb{N}$. We claim for all $k\in \mathbb{N}$, 
\begin{align*} 
&\psi(\sigma_4\sigma_1\sigma_2\sigma_4\sigma_5^{-1}\sigma_4^{-1}\sigma_3\sigma_5\sigma_1^{-1}\sigma_2(\sigma_2\sigma_3)^k) \\ 
&= d(\sigma_4\sigma_1\sigma_2\sigma_4\sigma_5^{-1}\dot{\sigma_4}^{-1}\sigma_3\sigma_5\dot{\sigma}_1^{-1}\sigma_2(\sigma_2\sigma_3)^k+\sigma_4\sigma_1\sigma_2\sigma_4\sigma_5^{-1}\sigma_4^{-1}\sigma_3\dot{\sigma}_5\dot{\sigma}_1^{-1}\sigma_2(\sigma_2\sigma_3)^k).
\end{align*} 
All the circles of all these states are marked with a $-$.
\end{enumerate} 
\subsection{A four-braid example} \label{subsec:4braidex}

Proposition \ref{4braidexample} follows from applying functoriality and stability to $\alpha_k = \triangle^2 \sigma_2^{-k}$ and $\eta_k = \triangle^2(\sigma_2\sigma_3)^{-k}$ in Table \ref{table:qprv} and the following theorem. 

\begin{thm}\label{4braid} The $4$-braid family 
$$\beta_{k} = \Delta^{2}\sigma_{3}^{-k}$$
where $k \in \mathbb{N}$ satisfies
$$\psi'(\beta_{k}) \neq 0.$$
\end{thm}

Notice that for $k > 12$, $\beta_{k}$ is not quasipositive since the writhe is $12 - k < 0$, 
so Theorem \ref{4braid} gives an infinite family of non-quasipositive braids with non-vanishing $\psi'$.
\begin{proof}

First, using Baldwin's computer program, we determine that 
$$\psi'(\beta_{9})\neq 0.$$
By functoriality this guarantees that $\psi'(\beta_{k})\neq 0$ for all $1 \leq k < 9$. We will induct on $k$ for all $k>9$.

Similar to the notation introduced in Section \ref{sec:genstability}, let $D_{0}^{k}$ and $D_{1}^{k}$ be the knots or links that are obtained by replacing the last crossing of $\beta_{k}$ with its 0- and 1-resolutions, respectively, and taking the closure, where $D_1^{k} = \widehat{\beta_{k-1}}$. We orient $D_{1}^{k}$ with the same orientation as $\widehat{\beta_{k-1}}$ (all strands oriented downwards). We orient $D_{0}^{k}$ so that all three outer strands are oriented downwards above the braid word. We first observe that for all $k > 1$, $D_{0}^{k}$ is isotopic to the disjoint union of the unknot, oriented counter-clockwise, and the Hopf link $\sigma_{1}^{2}$. So a quick computation shows that the reduced Khovanov homology over $\mathbb{Z}/2\mathbb{Z}$ of $D_{0}^{k}$ is as shown in Table \ref{table:Kh}.

\begin{table}
\begin{center}
 \begin{tabular}{||c | c | c | c||} 
 \hline
$j$ & $i=0$ & $i=2$ \\ [0.5ex] 
 \hline\hline
 0 & $\mathbb{Z}/2\mathbb{Z}$  &  \\ 
 \hline
 2 &  $\mathbb{Z}/2\mathbb{Z}$  &   \\
 \hline
 4 & & $\mathbb{Z}/2\mathbb{Z}$ \\
 \hline
 6 &   & $\mathbb{Z}/2\mathbb{Z}$  \\[1ex] 
 \hline

\end{tabular}
\caption{The reduced Khovanov homology over $\mathbb{Z}/2\mathbb{Z}$ of $D_{0}^{k}$.}\label{table:Kh}
\end{center}
\end{table}

In addition, the number of negative crossings in $D_{0}^{k}$ is $6$ for all $k > 1$. We are interested
in $\overline{Kh}(\widehat{\beta_{k}})$ in homological grading $0$ and $q$-grading the self-linking number of $\beta_{k}$ plus one, so $i = 0$ and $j = -4+12-k+1 = 9-k$. The long exact sequence \eqref{eq:ssesn} for reduced Khovanov homology corresponding to taking the resolution of the last negative crossing in the word then takes the following form:
$$ \cdots \longrightarrow \overline{Kh}^{k-7}_{2k-10}(D_{0}^{k}) \longrightarrow \overline{Kh}^{0}_{10-k}(D_{1}^{k}) \longrightarrow \overline{Kh}^{0}_{9-k}(\widehat{\beta_{k}}) \longrightarrow \overline{Kh}^{k-6}_{2k-10}(D_{0}^{k}) \longrightarrow \cdots,$$
where the $u$ from \eqref{eq:ssesn} is $u = n_{-}(D_{0}^{k}) - n_{-}(\widehat{\beta_{k}}) = 6 - k$. 
For $k\geq 10$, both $k-7,  k-6>2$. Using the information on the reduced Khovanov homology over $\Z/2\Z$ of $D_0^k$, the long exact sequence becomes
$$ \cdots \longrightarrow 0 \longrightarrow \overline{Kh}^{0}_{10-k}(D_{1}^{k}) \longrightarrow \overline{Kh}^{0}_{9-k}(\widehat{\beta_{10}}) \longrightarrow 0 \longrightarrow \cdots$$
 The map on chain complexes yields an isomorphism 
$$\overline{Kh}^0_{10-k}(\widehat{\beta_{k-1}}) \cong \overline{Kh}^{0}_{10-k}(D_{1}^{k}) \\\stackrel{\cong}{\longrightarrow} \overline{Kh}^{0}_{9-k}(\widehat{\beta_{k}}).$$ This isomorphism is induced by the map that naturally sends $\widetilde{\psi}'(\beta_{k-1})$ to $\widetilde{\psi}'(\beta_{k})$. Hence since $\psi'(\beta_9) \in \overline{Kh}^{0}_{0}(\widehat{\beta_{9}})$ is non-zero as computed earlier, the isomorphism implies that $\psi'(\beta_k) \in \overline{Kh}^{0}_{{9-k}}(\widehat{\beta_{k}})$ is non-zero for all $k\geq 10$. 

\end{proof}

\section{Bennequin-type inequalities and the maximum self-linking number}\label{maxselflinking}

In this section we prove Theorem \ref{thm:homflypretzel}. We will first give the necessary background on the maximal self-linking number and recall some results which bound the maximal self-linking number using the HOMFLY-PT polynomial of the link. We follow the conventions of the Knot Atlas \cite{kat} for the HOMFLY-PT polynomial. 

Recall the self-linking number $\Sl(\overline{L})$ of a transverse link $\overline{L}$ as defined in Definition \ref{defn:selflinking}.
\begin{defn} \label{d.msl} The \emph{maximal self-linking number}, $\msl(L)$ of a smooth link $L$ is the maximum of $\Sl (\overline{L})$ taken over all transverse link representatives $\overline{L}$ of $L$. 
\end{defn}

Let $P_L(a, z)$ be the HOMFLY-PT polynomial of a smooth link $L$, normalized so that $P=1$ for the unknot and defined by the following skein relation.
\begin{equation} \label{eq:homfly}
\vcenter{\hbox{
\def\svgwidth{.5\columnwidth}
\begingroup%
  \makeatletter%
  \providecommand\color[2][]{%
    \errmessage{(Inkscape) Color is used for the text in Inkscape, but the package 'color.sty' is not loaded}%
    \renewcommand\color[2][]{}%
  }%
  \providecommand\transparent[1]{%
    \errmessage{(Inkscape) Transparency is used (non-zero) for the text in Inkscape, but the package 'transparent.sty' is not loaded}%
    \renewcommand\transparent[1]{}%
  }%
  \providecommand\rotatebox[2]{#2}%
  \ifx\svgwidth\undefined%
    \setlength{\unitlength}{338.90498047bp}%
    \ifx\svgscale\undefined%
      \relax%
    \else%
      \setlength{\unitlength}{\unitlength * \real{\svgscale}}%
    \fi%
  \else%
    \setlength{\unitlength}{\svgwidth}%
  \fi%
  \global\let\svgwidth\undefined%
  \global\let\svgscale\undefined%
  \makeatother%
  \begin{picture}(1,0.05241179)%
    \put(0,0){\includegraphics[width=\unitlength,page=1]{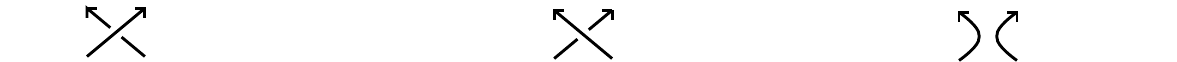}}%
    \put(-0.00235132,0.02862324){\color[rgb]{0,0,0}\makebox(0,0)[lb]{\smash{$aP_{}$}}}%
    \put(0.29574238,0.02702738){\color[rgb]{0,0,0}\makebox(0,0)[lb]{\smash{$-a^{-1}P_{}$}}}%
    \put(0.683218,0.0255099){\color[rgb]{0,0,0}\makebox(0,0)[lb]{\smash{$=zP_{}$}}}%
    \put(0.55540221,0.23779022){\color[rgb]{0,0,0}\makebox(0,0)[lt]{\begin{minipage}{0.56315827\unitlength}\raggedright \end{minipage}}}%
    \put(0.1261297,0.0255099){\color[rgb]{0,0,0}\makebox(0,0)[lb]{\smash{$(a, z)$}}}%
    \put(0.52742212,0.0255099){\color[rgb]{0,0,0}\makebox(0,0)[lb]{\smash{$(a, z)$}}}%
    \put(0.86734041,0.0255099){\color[rgb]{0,0,0}\makebox(0,0)[lb]{\smash{$(a, z)$}}}%
  \end{picture}%
\endgroup%
 }}
\end{equation}
The pictures $\vcenter{\hbox{\def \svgwidth{.03\columnwidth} 
\begingroup%
  \makeatletter%
  \providecommand\color[2][]{%
    \errmessage{(Inkscape) Color is used for the text in Inkscape, but the package 'color.sty' is not loaded}%
    \renewcommand\color[2][]{}%
  }%
  \providecommand\transparent[1]{%
    \errmessage{(Inkscape) Transparency is used (non-zero) for the text in Inkscape, but the package 'transparent.sty' is not loaded}%
    \renewcommand\transparent[1]{}%
  }%
  \providecommand\rotatebox[2]{#2}%
  \ifx\svgwidth\undefined%
    \setlength{\unitlength}{17.5144927bp}%
    \ifx\svgscale\undefined%
      \relax%
    \else%
      \setlength{\unitlength}{\unitlength * \real{\svgscale}}%
    \fi%
  \else%
    \setlength{\unitlength}{\svgwidth}%
  \fi%
  \global\let\svgwidth\undefined%
  \global\let\svgscale\undefined%
  \makeatother%
  \begin{picture}(1,0.86277612)%
    \put(0,0){\includegraphics[width=\unitlength,page=1]{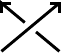}}%
  \end{picture}%
\endgroup%
}}$, $\vcenter{\hbox{\def \svgwidth{.03\columnwidth} 
\begingroup%
  \makeatletter%
  \providecommand\color[2][]{%
    \errmessage{(Inkscape) Color is used for the text in Inkscape, but the package 'color.sty' is not loaded}%
    \renewcommand\color[2][]{}%
  }%
  \providecommand\transparent[1]{%
    \errmessage{(Inkscape) Transparency is used (non-zero) for the text in Inkscape, but the package 'transparent.sty' is not loaded}%
    \renewcommand\transparent[1]{}%
  }%
  \providecommand\rotatebox[2]{#2}%
  \ifx\svgwidth\undefined%
    \setlength{\unitlength}{17.5144927bp}%
    \ifx\svgscale\undefined%
      \relax%
    \else%
      \setlength{\unitlength}{\unitlength * \real{\svgscale}}%
    \fi%
  \else%
    \setlength{\unitlength}{\svgwidth}%
  \fi%
  \global\let\svgwidth\undefined%
  \global\let\svgscale\undefined%
  \makeatother%
  \begin{picture}(1,0.83575131)%
    \put(0,0){\includegraphics[width=\unitlength,page=1]{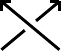}}%
  \end{picture}%
\endgroup%
}}$, and $\vcenter{\hbox{\def \svgwidth{.03\columnwidth} 
\begingroup%
  \makeatletter%
  \providecommand\color[2][]{%
    \errmessage{(Inkscape) Color is used for the text in Inkscape, but the package 'color.sty' is not loaded}%
    \renewcommand\color[2][]{}%
  }%
  \providecommand\transparent[1]{%
    \errmessage{(Inkscape) Transparency is used (non-zero) for the text in Inkscape, but the package 'transparent.sty' is not loaded}%
    \renewcommand\transparent[1]{}%
  }%
  \providecommand\rotatebox[2]{#2}%
  \ifx\svgwidth\undefined%
    \setlength{\unitlength}{17.5144927bp}%
    \ifx\svgscale\undefined%
      \relax%
    \else%
      \setlength{\unitlength}{\unitlength * \real{\svgscale}}%
    \fi%
  \else%
    \setlength{\unitlength}{\svgwidth}%
  \fi%
  \global\let\svgwidth\undefined%
  \global\let\svgscale\undefined%
  \makeatother%
  \begin{picture}(1,0.83706914)%
    \put(0,0){\includegraphics[width=\unitlength,page=1]{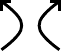}}%
  \end{picture}%
\endgroup%
}}$ indicate smooth links $L_+, L_-$, and $L_0$ where $L_+$ and $L_-$ differ by switching a crossing, and $L_0$ is the link resulting from choosing the oriented resolution at the crossing.  

By \cite{FW87} and \cite{Mor86}, we have the following inequality.
\begin{thm}{(\cite{FW87}, \cite{Mor86})} \label{thm:mslbhomfly}
\[\msl(L) \leq -\deg_a(P_L(a, z))-1, \] 
where $\deg_a(P_L(a, z))$ is the maximum degree in $a$ of $P_L(a, z)$.
\end{thm} Ng also provides a skein-theoretic proof that unifies several similiar inequalities in \cite{Ng08}. The transverse element $\widetilde{\psi}(\beta)$ for a braid representative $\beta$ of $L$ is always supported in the grading $i=0$ and $j=\Sl(\hat{\beta})$ in $Kh(L)$ \cite[Proposition 2]{Pla06}. Recall that by Remark \ref{rem:msl0}, this means that if $Kh^0_j(L) = 0$ for all $j \leq \msl(L)$, then $\psi(\beta) = 0$ for every braid representative $\beta$ of $L$.  

Consider the 3-tangle pretzel knots $K=P(r, -s, -t)$ where $r>0$ is even and $s, t> 0$ are odd. Our convention is illustrated in Figure \ref{fig:pconvention} below. 
\begin{figure}[ht]
\def\svgwidth{.1\columnwidth}
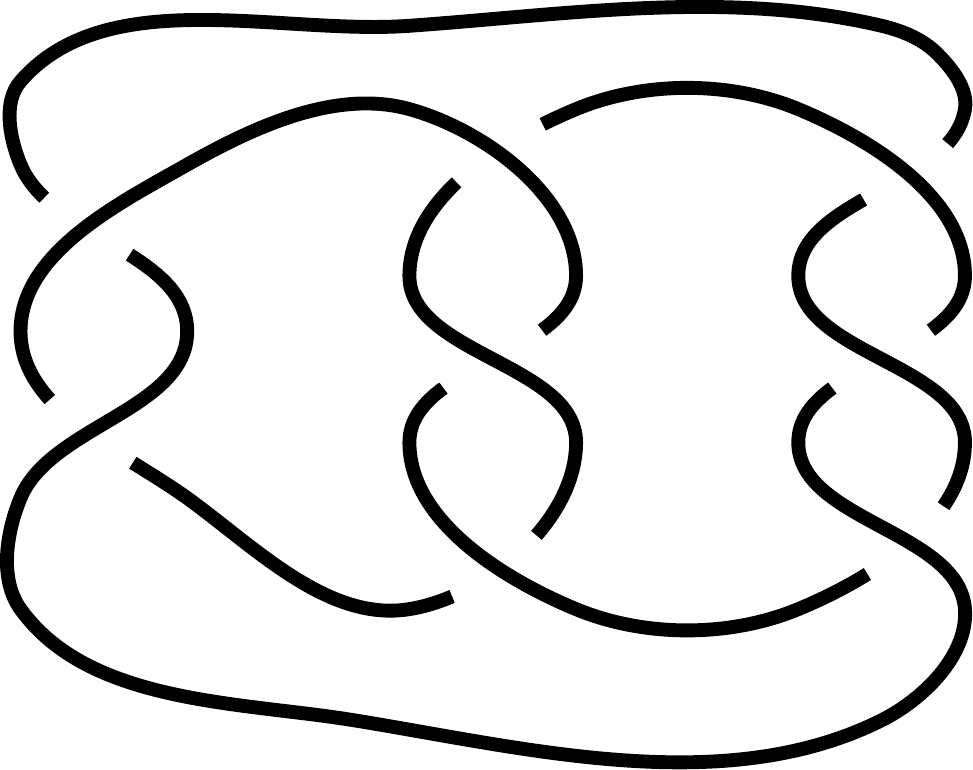
\caption{\label{fig:pconvention} The $P(2, -3, -3)$ pretzel knot.}
\end{figure} 
Since $K$ is a negative knot by the standard pretzel diagram $D$, which means that all the crossings are negative in $D$ with an orientation, there is a single state, the all-1 state, which chooses the $1$-resolution on all the crossings of $D$ and gives the generators for the chain complex at homological grading $i=0$. Recall $|s_1(D)|$ is the number of state circles in the all-$1$ state. Since $K$ is a negative knot, the state graph $s_1(D)$ has no one-edged loops, so $K$ is adequate on one side by definition. It is known, see for example the proof of \cite[Proposition 5.1]{Kho03}, that this implies that in $i=0$, there are only two possible nontrivial homology groups at $j=|s_1(D)|-n_-(D)$ and $j = |s_1(D)|-n_-(D)-2$. It is also possible to see this for these pretzel knots by direct computation. Note that Manion gives an explicit characterization of the Khovanov homology of 3-tangle pretzel knots in \cite{Man14}. 

Now 
\[|s_1(D)| = r+1. \]
Thus 
\[ |s_1(D)|-n_-(D)= 1-s-t,  \]
and $Kh(K)$ can only have nontrivial homology groups for $i=0$ at $j = 1-s-t, 1-s-t-2$.

Before proving Theorem \ref{thm:homflypretzel}, it is helpful to see an example in the $P(2, -5, -5)$ pretzel knot. 

\begin{eg} \label{eg:n255}
The pretzel knot $K=P(2, -5, -5)$. 
$Kh(K)$ has nontrivial homology groups supported in $i=0, j=-11$ and $i=0, j=-9$, and trivial homology groups for all other $j$ when $i=0$. 

It has HOMFLY-PT polynomial \cite{kat}
\begin{align*} P_K(a, z) &= 10 a^{10} - 13 a^{12} + 4 a^{14} + 39 a^{10} z^{2} - 32 a^{12} z^2 + 4 a^{14} z^2 + 
 57 a^{10} z^4 - 27 a^{12} z^4 + a^{14} z^4  \\
 &+36 a^{10} z^6 - 9 a^{12} z^6 + 10 a^{10} z^8 - a^{12} z^8 + a^{10} z^{10}. \end{align*} 
 
Using Theorem \ref{thm:mslbhomfly} gives that $\msl(P(2, -5, -5)) \leq -(14)-1 < -11.$ Therefore, $\psi=0$ for every braid representative of $P(2, -5, -5)$. 
\end{eg}

We generalize the above examples using the computation for the HOMFLY-PT polynomial for torus knots by Jones \cite{Jon87}. For our purpose it is enough to have the following lemma which we prove here by inducting on the defining skein relation.
\begin{lem} \label{lem:toruspqhomfly} Let $T_{2, -q}$ denote the negative $2q$ torus link with all negative crossings. If $q>1$ is odd, then 
\[\deg_a(P_{T_{2, -q}}(a, z)) = q+1, \] with all negative coefficients. 
If $q>1$ is even, then with the orientation given that makes all the crossings negative,
\[\deg_a(P_{T_{2, -q}}(a, z)) = q+1, \]  with all positive coefficients.
\end{lem}

\begin{proof} We give a proof here by induction. The base cases are $q=2$ and $q=3$. We see respectively \cite{kat} that $P_{T_{2, -2}}(a, z)$ has all positive coefficients with the term of the highest $a$-degree given by $+\frac{a^3}{z}$. 
Similarly, $P_{T_{2, -3}}(a, z)$ has all negative coefficients with the term of the highest $a$-degree: $-a^4$.   
For $P_{T_{2, -q}}(a, z)$, where $q>3$, we expand a single crossing by \eqref{eq:homfly}. This gives that 
\begin{equation} \label{eq:thomfly}
P_{T_{2, -q}} = a^2 P_{T_{2, -(q-2)}} - azP_{T_{2, -(q-1)}}.
\end{equation}
Assuming the induction hypothesis, we have $\deg_a P_{T_{2, -(q-2)}} (a, z) = q-1$ with all positive/negative coefficients for even/odd $q-2$. Similarly, we have 
$\deg_a P_{T_{2, -(q-1)}}(a, z) = q$ with all negative/positive leading coefficients for odd/even $q-1$. Plugging this into \eqref{eq:homfly} gives that there is no cancellation between the terms with the maximal $a$-degree $q+1$, and the coefficients are either all positive when $q$ is even, or all negative when $q$ is odd.
\end{proof} 

Now we show Theorem \ref{thm:homflypretzel}, which we reprint here for reference.
\begin{restate} 
Let $K=P(r, -q,  -q)$ be a pretzel knot with $q>0$ odd and $r\geq 2$ even, then $\psi = 0$ for every transverse link representative of $K$. 
\end{restate}

\begin{proof}
We apply relation \eqref{eq:homfly} to the top left negative crossing of $P(r, -q, -q)$. Denote the diagram obtained by switching the crossing by $D_+$ and the diagram obtained by resolving the crossing following the orientation by $D_0$. See Figure \ref{fig:npretzel}.
\begin{figure}[ht]
\def\svgwidth{.5\columnwidth}
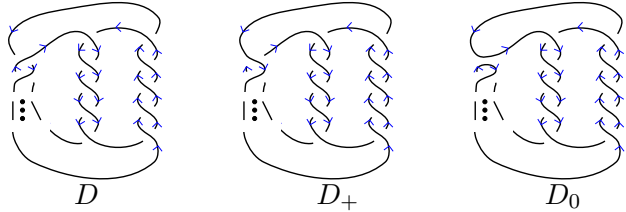
\caption{\label{fig:npretzel} Link diagrams appearing in the relation for the HOMFLY-PT polynomial}
\end{figure} 
 Then we have  
\[ a^{-1}(P_D(a, z)) = a(P_{D_+}(a, z)) -z(P_{D_0}(a, z)),\]
and $D_0$ is the diagram of the torus link $T_{2, -2q}$ with the orientation as in Figure \ref{fig:npretzel}. By Lemma \ref{lem:toruspqhomfly} we know that $\deg_a P_{T_{2, -2q}}(a, z) = 2q+1$. 

We first consider $P(2, -q, -q)$. Switching the top left negative crossing results in $D_+$ being a connected sum of 2 $T_{2, -q}$'s, and $D_0$ is $T_{2, -2q}$. Therefore $\deg_a P_{D_+}(a, z)=2q+2$ as the HOMFLY-PT polynomial of a connected sum is the product of the individual HOMFLY-PT polynomials, and $\deg_a P_{D_0}(a, z) = 2q+1$. This clearly shows 
\[\deg_a P_D(a, z) = 2 + \deg_aP_{D_+}(a, z)=2q+4.  \] 

For even $r>2$ we may now induct on $r$ with the hypothesis that 
\[\deg_a(P_{P(r, -q, -q)}(a, z)) = 2+r+2q. \] 
  Indeed, notice that switching the top left negative crossing of $P(r,-q,-q)$ yields that $D_{+}$ is simply $P(r-2,-q,-q)$ and that $D_{0}$ is still $T_{2,-2q}$. Thus we obtain that
\[\msl(P(r, -q, -q)) \leq  -\deg_a(P_K(a, z))-1 \leq -2-r-2q-1  \]
by Theorem \ref{thm:mslbhomfly}. 

On the other hand, there are only two possible nontrivial homology groups for $i=0$ with $j$-grading equal to $|s_1(D)|-n_-(D) = 1-2q$ and $-1-2q$ in $Kh(P(r, -q, -q))$. We apply Remark \ref{rem:msl0} to finish the proof of the theorem.
\end{proof} 

Recall that any transverse link $L$ with a quasipositive braid representative $\beta$ satisfies that $\psi(L) \neq 0$ \cite{Pla06}. Thus Theorem \ref{thm:homflypretzel} directly implies that no transverse representative of such $P(r,-q,-q)$ has a quasipositive braid representative. We can conclude that as a smooth link, $P(r,-q,-q)$ is not the closure of a quasipositive braid, and so:

\begin{cor}\label{cor:pretzelQP} Every $3$-tangle pretzel knot of the form $P(r,-q,-q)$ with $q>0$ odd and $r\geq 2$ even is not quasipositive. 
\end{cor}

Recall that $P(r,-q,-q)$ is a negative knot. We thank Peter Feller for the observation that another argument can be made to show that, in general, \emph{any} negative knot that is quasipositive must be the unknot (using the fact that any negative knot is strongly quasinegative \cite{nakamura_positive}, \cite{rudolph_stronglyQP} and facts about the behavior of the Ozsv\'ath-Szab\'o concordance invariant $\tau$ \cite{ozsvath_szabo_tau} for quasipositive and strongly quasinegative knots). Our proof method for Corollary \ref{cor:pretzelQP} is clearly of a very different flavor, as we do not depend on tools from Heegaard Floer homology or four-dimensional topology.  See also \cite{boileau_quasipositive} for a discussion of the strong quasipositivity of 3-tangle pretzel knots, up to mirror images.

To understand these examples better, we consider the FDTCs of some braid representatives of these pretzel knots which do not admit a transverse representative with non-vanishing $\psi$.
\begin{eg}  Table \ref{table:pretzel} gives some examples of pretzel knots satisfying the conditions of Theorem \ref{thm:homflypretzel} and braid representatives, using Hunt's program. \\

\begin{table}
\begin{tabular}{|p{2cm}|p{13cm}|}
\hline
  Knot & Braid representative \\
  \hline
  $P(2, -5, -5)$ & $\sigma_1^{-1} \sigma_2^{-5} \sigma_1^{-1} \sigma_2^{-5}$ \\
  \hline
   $P(4, -5, -5)$ & $\sigma_1^{-1} \sigma_2^{-5} \sigma_3^{-1} \sigma_2^{-5} \sigma_1 \sigma_2 \sigma_3^{-1} \sigma_4 \sigma_3^{-3} \sigma_2^{-1} \sigma_3^{-1} \sigma_4^{-1}$ \\ 
 \hline  
 $P(6, -5, -5)$ & \text{$\sigma_1^{-1} \sigma_2 \sigma_3^{-1} \sigma_4^{-1} \sigma_3^{-1} \sigma_2^{-1} \sigma_3^{-1} \sigma_4 \sigma_5 \sigma_6 \sigma_1 \sigma_2^{-1} \sigma_3^{-1}\sigma_4\sigma_5\sigma_4^{-5}\sigma_3^{-1}\sigma_4^{5}\sigma_5^{-1}\sigma_6^{-1}\sigma_2\sigma_3^{-1}\sigma_4^{-1}\sigma_5^{-1}$} \\
 \hline
 $P(8, -5, -5)$ & $\sigma_1^{-1}\sigma_2^{-1}\sigma_3\sigma_4^{-1}\sigma_5^{-1}\sigma_6^{-1}\sigma_7^{-1} \sigma_4^{-1}\sigma_5^{-5}\sigma_6^{-1}\sigma_3^{-1} \sigma_4^{-1}\sigma_5^{-1}\sigma_4^{-1}\sigma_2\sigma_3^{-1}\sigma_4^{-1}\sigma_5\sigma_6\sigma_7\sigma_8
\sigma_5^{-5}$ \\
& $\sigma_6 \sigma_7\sigma_5\sigma_6\sigma_1\sigma_2\sigma_3^{-1}\sigma_4^{-1}\sigma_5^{-1}\sigma_4^{-1}\sigma_3\sigma_4^{-1}\sigma_5\sigma_6^{-1}\sigma_7^{-1}\sigma_8^{-1}\sigma_2^{-1}\sigma_3$ \\
\hline

\end{tabular}
\vspace{1mm}
\caption{Examples of pretzel knots satisfying the conditions of Theorem \ref{thm:homflypretzel}, together with braid word representatives.}\label{table:pretzel}
\end{table}
\end{eg}

\begin{rem}
Each of these braid representatives has at most a single $\sigma_{1}$ and a single $\sigma_{1}^{-1}$. By Proposition \ref{FDTCestimate}, this implies that each of their FDTCs lies in the interval $[-1,1]$. Notice that every transverse link has \emph{some} braid representative with FDTC in $[-1,1]$, since any $n$-braid that is a positive stabilization of some $(n-1)$-braid has FDTC lying in $[0,1]$. 
\end{rem} 

\begin{que}\label{psizerobraidrep1} Suppose $K$ is a smooth link such that $\psi(\beta) = 0$ for all braid representatives $\beta$ of $K$. Then is the FDTC of each braid representative of $K$ less than or equal to one?
\end{que}

An affirmative answer to Question \ref{psizerobraidrep1} would prove a statement similar in flavor to Theorem \ref{FDTCPlamenevskaya} in the setting of Khovanov homology. In particular, its contrapositive would state that if a link $K$ has some braid representative whose FDTC is strictly greater than one, then it has some transverse representative for which $\psi$ does not vanish.

\bibliographystyle{amsalpha}\bibliography{reference}

\end{document}